%% file: 0main.tex
\documentclass[12pt]{article}
\usepackage[margin=1in]{geometry}

\usepackage{tikz-cd}
\usepackage{tikz}
\usepackage{adjustbox}
\usepackage[algo2e,ruled,vlined]{algorithm2e}
\usepackage[normalem]{ulem}
\usepackage{multirow}
\usepackage{amsmath}
\usepackage{amssymb}
\usepackage{amsthm}
\usepackage{authblk}
\usepackage{graphicx}
\usepackage{fullpage}
\usepackage{subcaption}
\usepackage{hyperref}

\long\def\remove#1{}

\DeclareMathOperator{\inv}{\sf Inv}

\DeclareMathOperator{\cl}{\sf cl}
\DeclareMathOperator{\mo}{\sf mo}
\DeclareMathOperator{\pf}{\sf pf}

\newcommand{\cV}{\mathcal{V}}

\definecolor{colorblind-yellow}{RGB}{221,170,51}
\definecolor{colorblind-red}{RGB}{187,85,102}
\definecolor{colorblind-blue}{RGB}{0,68,136}
\definecolor{dark-gray}{RGB}{64,64,64}
\definecolor{medium-gray}{RGB}{114,114,114}
\definecolor{light-gray}{RGB}{190,190,190}

\newtheorem{theorem}{\sffamily Theorem}
\newtheorem{lemma}[theorem]{\sffamily Lemma}

\newtheorem{proposition}[theorem]{\sffamily Proposition}
\newtheorem{definition}[theorem]{\sffamily Definition}

\newif\ifpaper
\papertrue

\title{Tracking Dynamical Features via Continuation and Persistence} 

\begin{document}

\author[1]{Tamal K. Dey\thanks{tamaldey@purdue.edu} }

\author[2]{Micha\l{} Lipi\'nski\thanks{michal.lipinski@uj.edu.pl}}

\author[2]{Marian Mrozek\thanks{marian.mrozek@uj.edu.pl}}

\author[1]{Ryan Slechta\thanks{rslechta@purdue.edu}}

\affil[1]{Department of Computer Science,
Purdue University, West Lafayette, USA}
\affil[2]{Division of Computational Mathematics, Faculty of Mathematics and Computer Science, Jagiellonian University, Krak\'{o}w, Poland}
\date{}

\maketitle

\setcounter{page}{1}
\begin{abstract}
Multivector fields and combinatorial dynamical systems have recently become a subject of interest due to their potential for use in computational methods. In this paper, we develop a method to track an isolated invariant set---a salient feature of a combinatorial dynamical system---across a sequence of multivector fields. This goal is attained by placing the classical notion of the ``continuation'' of an isolated invariant set in the combinatorial setting. In particular, we give a ``Tracking Protocol'' that, when given a seed isolated invariant set, finds a canonical continuation of the seed across a sequence of multivector fields. In cases where it is not possible to continue, we show how to use zigzag persistence to track homological features associated with the isolated invariant sets. This construction permits viewing continuation as a special case of persistence. 
\end{abstract}

\input{1intro}
\input{2prelim}
\input{3trackingprotocol}
\input{4continuation}
\input{5general}
\input{6conclusion}

\bibliographystyle{abbrv}
\bibliography{refs}

\end{document}

%% file: 1intro.tex
\section{Introduction}
\label{sec:intro}

Dynamical systems enter the field of data science in two ways: either directly, as in the case of dynamic data, 
or indirectly, as in the case of images, where gradient dynamics are useful.
Forman's discrete Morse theory \cite{Forman1998b,Forman1998a,K2015} combines topology with gradient dynamics via \emph{combinatorial vector fields}.
Discrete Morse theory has been used to simplify datasets and to extract topological features from them \cite{AKLM2019,DW22,KKM2005,RWS2011}. When coupled with persistent homology \cite{DW22,EH2010,ELZ02},
this theory can be useful 
for analyzing complex data~\cite{GRHW2011,KW21,landi2021}.

Conley theory \cite{Co78} is a generalization of classical Morse theory beyond gradient dynamics.
Conley's approach to dynamical systems is motivated by the observation that in many areas, perhaps most notably biology, the differential equations 
governing systems of interest are known only roughly. Generally, this is due to the presence of several parameters which cannot be measured or estimated precisely. 
A similar situation occurs in data science, where a time series dataset that is collected from a dynamical process only crudely approximates the underlying system.  This observation has motivated recent studies \cite{BKMW2020,DMS2021,KMW2016,LKMW2020,Mr2017,MW2021} on a variant of Conley theory for combinatorial  vector fields. 

The primary objects of interest in Conley theory are \emph{isolated invariant sets}, each
of which is a salient feature of a vector field, together with an associated homological invariant called the \emph{Conley index} (see Section \ref{sec:prelim}). 
Notably, isolated invariant sets with non-trivial Conley index persist under small perturbations. The geometry of the isolated invariant set may change, 
and even the topology may change, but the Conley index associated with the isolated invariant set remains the same. The isolated invariant set cannot suddenly vanish or change from an attractor to a repeller or vice versa.
From this observation, we get the notion of the \emph{continuation} of an isolated invariant set in a dynamical system to another one in a nearby system. 
This local idea becomes global by making the continuation relation transitive.  

Given a path in the space of dynamical systems, one can track an invariant set along the path so long as the invariant set remains isolated.
When the isolation is lost, continuation ``breaks'' and the Conley index is not well-defined. Typically, this may be observed when two isolated invariant sets merge. 
Isolation may eventually be regained, but there is no guarantee that the Conley index will be recovered. We propose to use persistence \cite{DW22,EH2010,ELZ02} in the discrete setting to connect continuations. First, we show how continuation can be detected and maintained algorithmically in a combinatorial multivector field, which is a discretized version of a continuous
vector field. In fact, the continuation itself may be viewed as a special case of persistence where all bars persist for the duration of the continuation. When continuation ``breaks,'' we observe the birth and death of homology classes.

The combination of continuation and persistence allows us to algorithmically track an isolated invariant set and its associated Conley index in the setting of combinatorial dynamical systems.
Recall that a combinatorial dynamical system is generated by a multivector field. A multivector field is a partition of a simplicial complex into sets that are convex with respect to the face poset. We track an isolated invariant set in a sequence of such fields
where each field differs from its adjacent ones by an \emph{atomic rearrangement}. Each \emph{atomic rearrangement} is either an \emph{atomic coarsening} or an \emph{atomic refinement}. We show that an atomic refinement always permits continuation and thus the
Conley index of the tracked invariant set persists. In the case of coarsening, we may not
be able to continue. In such a case, we select an isolated invariant set that is a minimal perturbation of the previous one and compute the persistence of the Conley index between them. Hence, while there may come a point where we can no longer track an isolated invariant set, we can use persistence to track the lifetime of the homological features that are associated with the isolated invariant set. 

\begin{figure}[htbp]
\begin{center}
  \begin{tabular}{c}
  \includegraphics[width=1.0\textwidth]{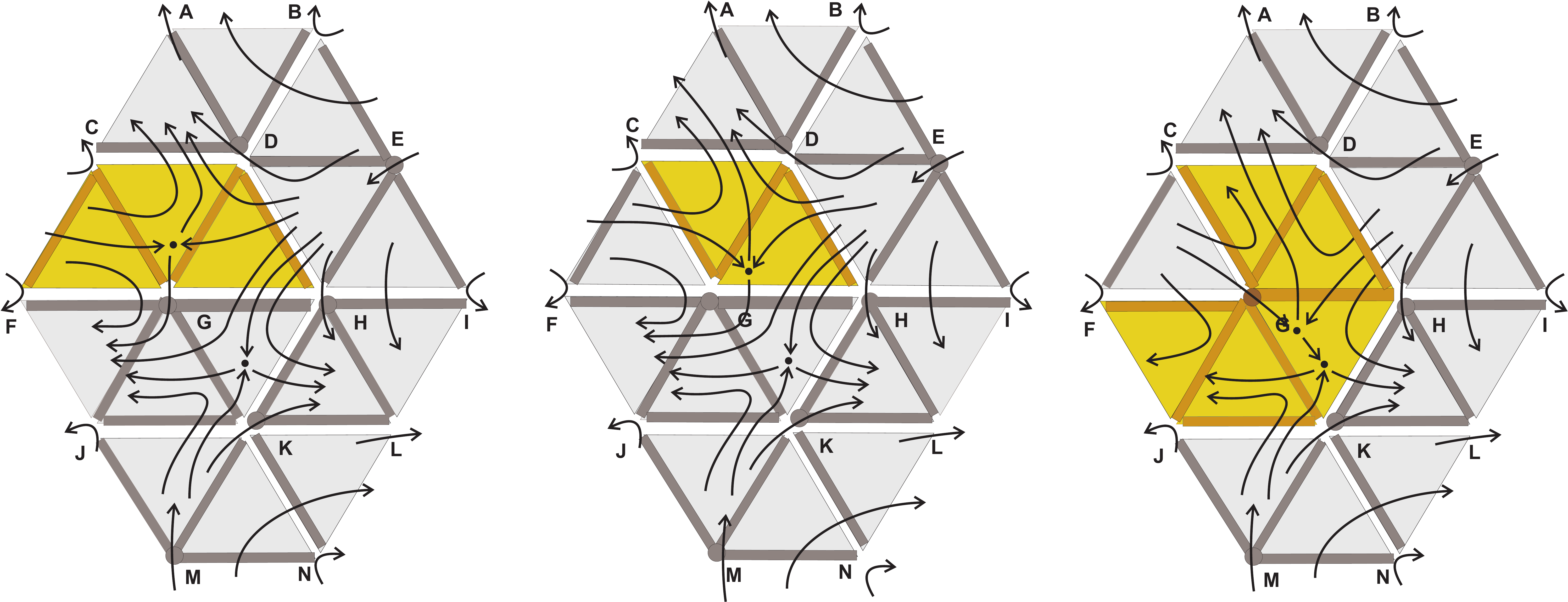}\\
  \begin{tikzpicture}[outer sep = 0, inner sep = 0]
  \draw[fill=light-gray,outer sep=0, inner sep=0] (-9.25,-0.3) -- (4.45,-0.3) -- (4.45,0.1) -- (-9.25,0.1) -- cycle;
  \node[align=left,color=black] at (-7.95,-0.1) {Dimension: 1};
  \draw[fill=light-gray,outer sep=0, inner sep=0] (0.5,-0.8) -- (4.45,-0.8) -- (4.45,-0.4) -- (0.5,-0.4) -- cycle;
  \node[align=left,color=black] at (1.8,-0.6) {Dimension: 1};
  \end{tikzpicture}\\

  \includegraphics[width=1.0\textwidth]{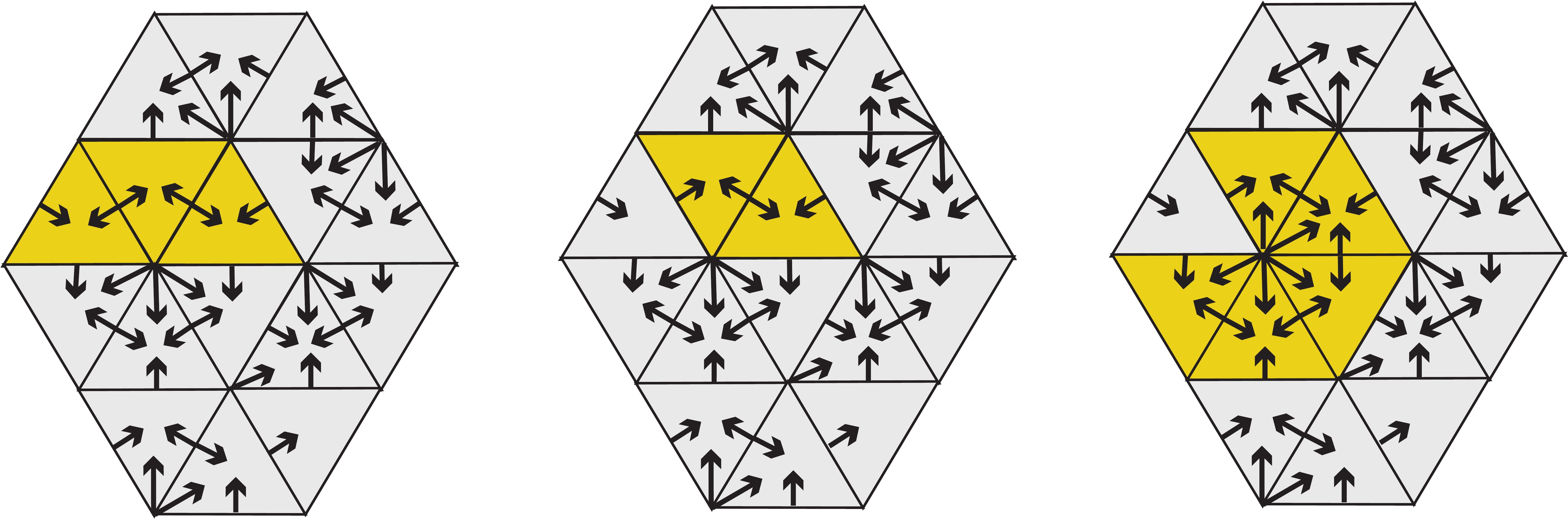}
\end{tabular}
\end{center}
  \caption{(Top) Three multivector fields, corresponding to merging saddles, where the middle multivector field is an atomic refinement of the left and the right multivector field is an atomic coarsening of the middle. The persistence barcode associated with the isolated invariant sets---depicted in yellow---is shown in gray below the three figures. 
 (Bottom)  The multivector fields associated with the figure at the top 
 using the standard multivector drawing convention.}
  \label{fig:Merging-saddles}
\end{figure}
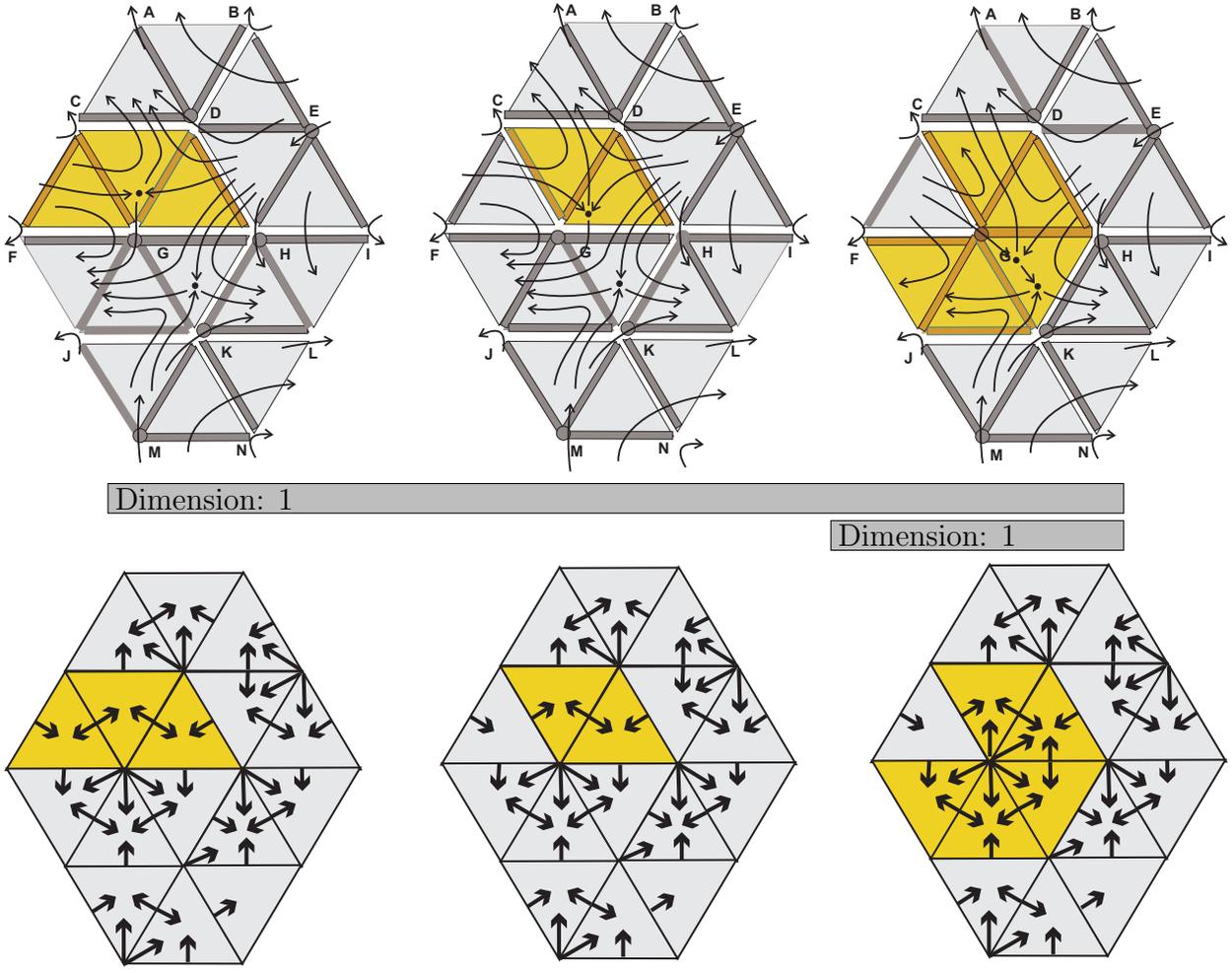

The top row of Figure~\ref{fig:Merging-saddles} presents flow lines from three flows on a simplicial complex with vertices marked from A to N.
The same figure also shows three combinatorial multivector fields represented as three different partitions of the collection 
of cells into multivectors. 
Each multivector is depicted as a connected component, and it is easy to see that they
are convex with respect to the face poset. The multivector fields are constructed as follows: if the flow transversely crosses an edge $e$ into a triangle $t$, then $e$ and $t$ are put in the same multivector. Else, $e$ is put into the same multivector as both of its incident triangles. If the flow line originating at a vertex $v$ immediately enters triangle $t$, then $v$ and $t$ are put into the same multivector. See \cite{MSTW21,MW2021} for additional information on this construction. 

There are two saddle stationary points in each flow, indicated by small black dots. For all three flows, the lower saddle is located in triangle
$GHK$. However, the upper one moves from triangle $CDG$ in the left flow, through triangle $DGH$ in the middle flow and finally it shares
triangle $GHK$ with the lower saddle in the right flow. On the combinatorial level, the upper saddle in the left flow is represented by an isolated invariant set $S_1$ consisting of one
multivector 
$\{CFG,CDG,DGH,CF,CG,DG,DH\}$
marked in yellow. The Conley index of $S_1$ is non-trivial only in dimension one
and has exactly one generator. 
Using methods from this paper, $S_1$ can be tracked to an isolated invariant set $S_2$
containing the upper saddle of the middle flow and consisting of one multivector $\{CDG,DGH,CG,DG,DH\}$, also marked in yellow.
The isolated invariant set $S_3$ containing  the upper saddle of the right flow consists of one multivector $\{CDG,DGH,FGJ,GJK,GHK,CG,DG,DH,FG,JK,GH,GJ,GK\}$, again marked in yellow. It is not a continuation of $S_2$, because in the right
flow the two saddles are too close to one another to be distinguishable with the resolution of the triangulation. Furthermore, the Conley index has changed. It is only nontrivial in dimension one, but unlike $S_1$ and $S_2$, it has two generators. Hence, in the right multivector field, a new generator is born. We show how to capture the birth of this generator using persistence, and we depict the associated barcode beneath the top row of Figure \ref{fig:Merging-saddles}. The familiar reader will note that the Conley index of $S_3$ is the same as that of a monkey saddle. However, because of the finite resolution, we cannot discriminate between two nearby saddles and a monkey saddle. One can view a multivector field as a combinatorial object that represents flows up to the resolution 
permitted by a triangulation. This purely combinatorial view of the top row of Figure \ref{fig:Merging-saddles} is presented in the bottom row. 
In subsequent examples, we use this style. 

%% file: 2prelim.tex
\section{Combinatorial Dynamical Systems}
\label{sec:prelim}

In this section, we review multivector fields, combinatorial dynamical systems, and isolated invariant sets. Throughout this paper, $K$ will always denote a finite simplicial complex. Furthermore, we will only consider simplicial homology \cite{hatcher,munkres} with coefficients taken from a finite field. Much of the foundational work on combinatorial dynamical systems was first published in \cite{Mr2017} and subsequently generalized in \cite{LKMW2020}. This work was heavily influenced by Forman's discrete Morse theory \cite{Forman1998b,Forman1998a}.  Combinatorial dynamical systems are constructed via multivector fields, which require a notion of convexity. Given a finite simplicial complex $K$, we let $\leq$ denote the \emph{face relation} on $K$. Formally, if $\sigma,\tau \in K$, then $\sigma \leq \tau$ if and only if $\sigma$ is a face of $\tau$. 
The set $A$ is \emph{convex} if for each pair $\sigma,\tau \in A$ where there exists a $\rho \in K$ satisfying $\sigma \leq \rho \leq \tau$, we have that $\rho \in A$. 

A \emph{multivector} is a convex subset of a simplicial complex. 
A partition of $K$ into multivectors is a \emph{multivector field} on $K$. Multivectors are not required to have a unique maximal element under $\leq$, nor are they required to be connected. Disconnected multivectors do not appear in practice, and in the interest of legibility, all examples that we include in this paper only depict connected multivectors. However, all of our theoretical results do hold for disconnected multivectors.
We draw a multivector $V$ by drawing an arrow from each nonmaximal element $\sigma \in V$ to each maximal element $\tau \in V$ where  $\sigma \leq \tau$. If $\sigma$ is the only element of a multivector, or a \emph{singleton}, then we mark $\sigma$ with a circle. Each $\sigma \in K$ is contained in a unique multivector $V \in \cV$. We denote the unique multivector in $\cV$ containing $\sigma$ as $[\sigma]_{\cV}$. 

A multivector field $\cV$ induces dynamics on $K$. Given a simplex $\sigma \in K$, we denote the \emph{closure} of $\sigma$ as $\cl(\sigma) := \{ \tau \in K \; | \; \tau \leq \sigma \}$. For a set $A \subseteq K$, the closure of $A$ is given by $\cl(A) := \cup_{\sigma \in A} \cl(\sigma)$. A set $A$ is \emph{closed} if and only if $A = \cl(A)$.  The multivector field $\cV$ induces a multivalued map $F_{\cV} \; : \; K \multimap K$ where $F_{\cV}(\sigma) := \cl(\sigma) \cup [\sigma]_{\cV}$. 
Informally, this will mean that if one is at a simplex $\sigma$, then one can move either to a face of $\sigma$ or to a simplex $\tau$ in the same multivector as $\sigma$. We allow moving within any single multivector because, on the level of flows, the behavior within the multivector is beyond the resolution of the given simplicial complex. Conversely, we do not allow moving from a cell to its coface, unless they are in the same multivector, because this does not respect the underlying flow.

The multivalued map $F$ gives a notion of \emph{paths} and \emph{solutions} to $\cV$. We let $\mathbb{Z}_{[i,j]} := \mathbb{Z} \cap [i,j]$. A \emph{path} is a function $\rho \; : \; \mathbb{Z}_{[0,n]} \to K$ where $\rho(i+1) \in F_{\cV}(\rho(i))$ for $i \in \mathbb{Z}_{[0,n-1]}$. 
Likewise, a \emph{solution} is a function $\rho \; : \; \mathbb{Z} \to K$ where $\rho(i+1) \in F_{\cV}(\rho(i))$. 
However, there are several trivial solutions in a multivector field. If $\sigma \in K$, then there is a solution $\rho$ where $\rho(i) = \sigma$ for all $i \in \mathbb{Z}$. 
That is, every simplex is a fixed point. This does not match the intuition from differential equations: only a very select set of simplices should be fixed points under $F_{\cV}$. To enforce this, we use the notion of a \emph{critical multivector}. But first, we define the \emph{mouth} of a set $A$, denoted $\mo(A)$, to be $\mo(A) := \cl(A) \setminus A$.  
The multivector $V$ is \emph{critical} if $H( \cl(A), \mo(A) ) \neq 0$.
Intuitively, critical multivectors with one maximal element correspond to stationary points in the flow setting.
Thus, only simplices in critical multivectors should be fixed points under $F$. \emph{Essential solutions} enforce this requirement \cite{LKMW2020}.
\begin{definition}[Essential Solution]
    Let $\rho \; : \; \mathbb{Z} \to K$ denote a solution to the multivector field $\cV$. 
    If for every $i \in \mathbb{Z}$ where $[ \rho(i) ]_{\cV}$ is not critical, there exists a pair of integers $i^- < i < i^+$ where $[\rho(i^-)]_{\cV} \neq [\rho(i)]_{\cV}$ and $[\rho(i)]_{\cV} \neq [\rho(i^+)]_{\cV}$, then $\rho$ is an \emph{essential solution}.
\end{definition}

\begin{figure}[htbp]
\centering
\begin{tabular}{ccc}
  \includegraphics[width=40mm]{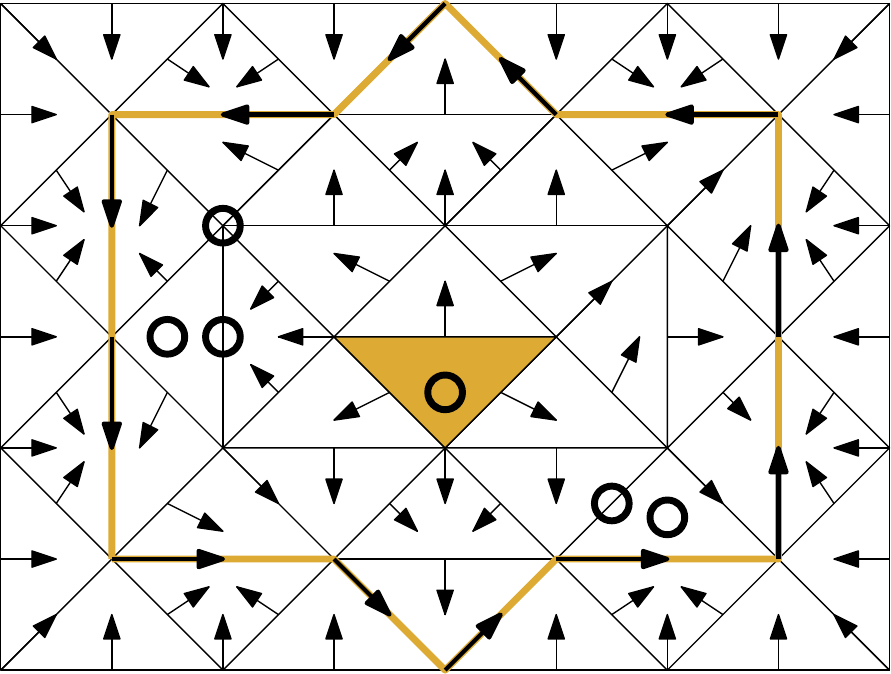}&
  \includegraphics[width=40mm]{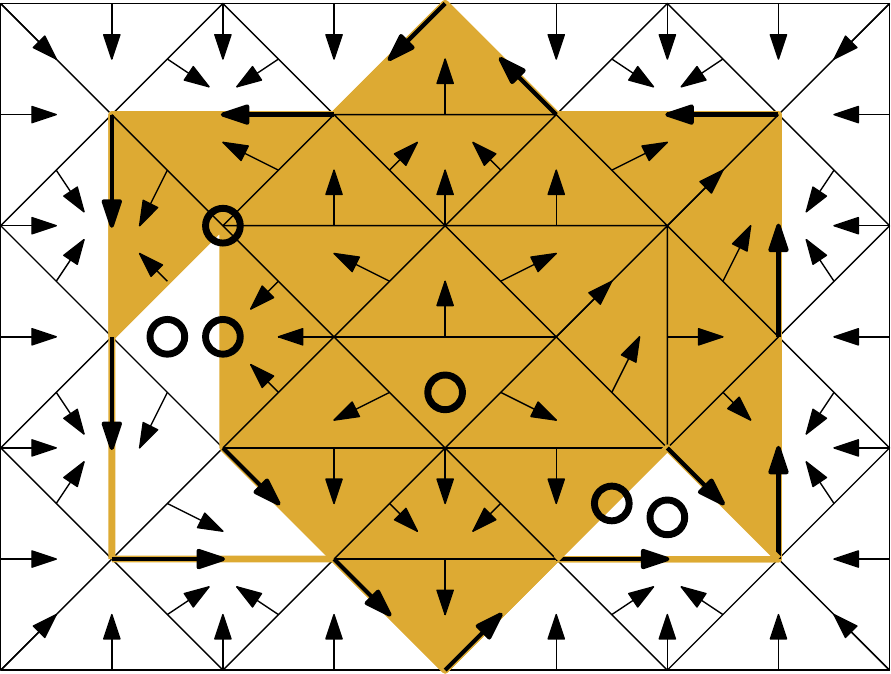}&
  \includegraphics[width=40mm]{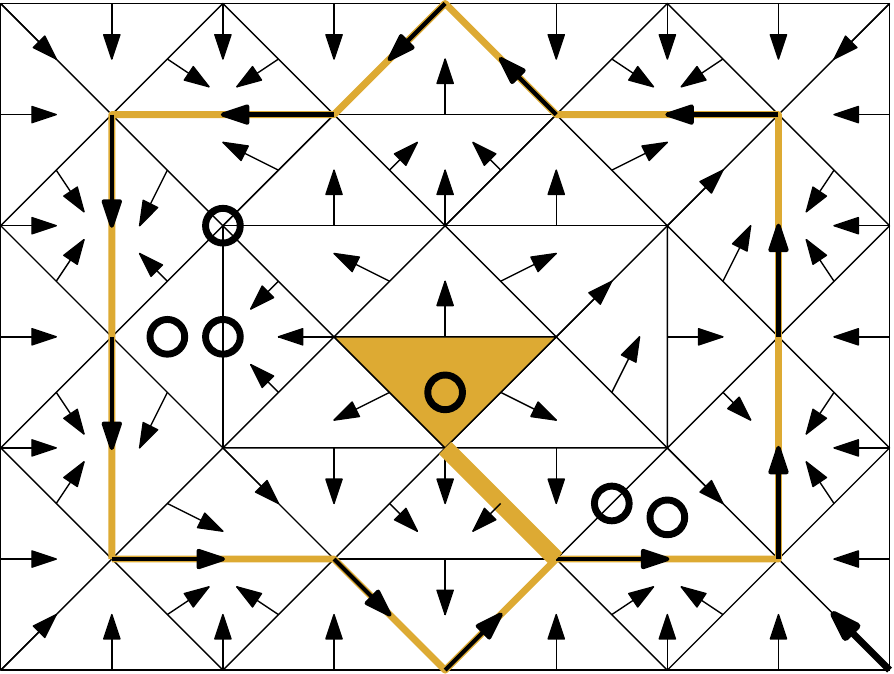}\\
\end{tabular}
\caption{Three examples of an invariant set, marked in yellow. }
\label{fig:invariant-sets}
\end{figure} 

\begin{figure}[htbp]
\centering
\begin{tabular}{ccc}
  \includegraphics[width=40mm]{fig/invariant-set-second-example.pdf}&
    \includegraphics[width=40mm]{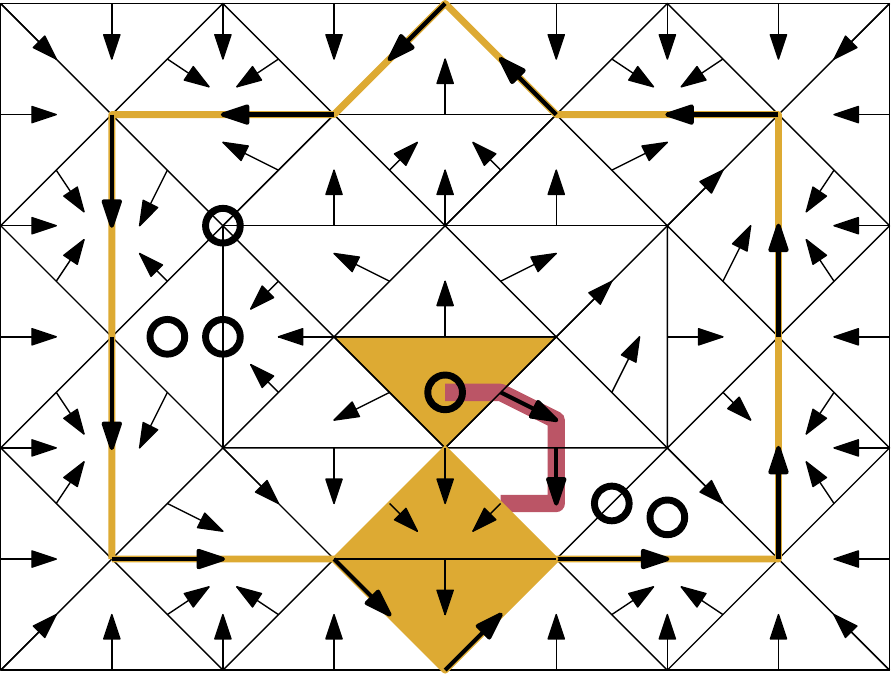}&
    \includegraphics[width=40mm]{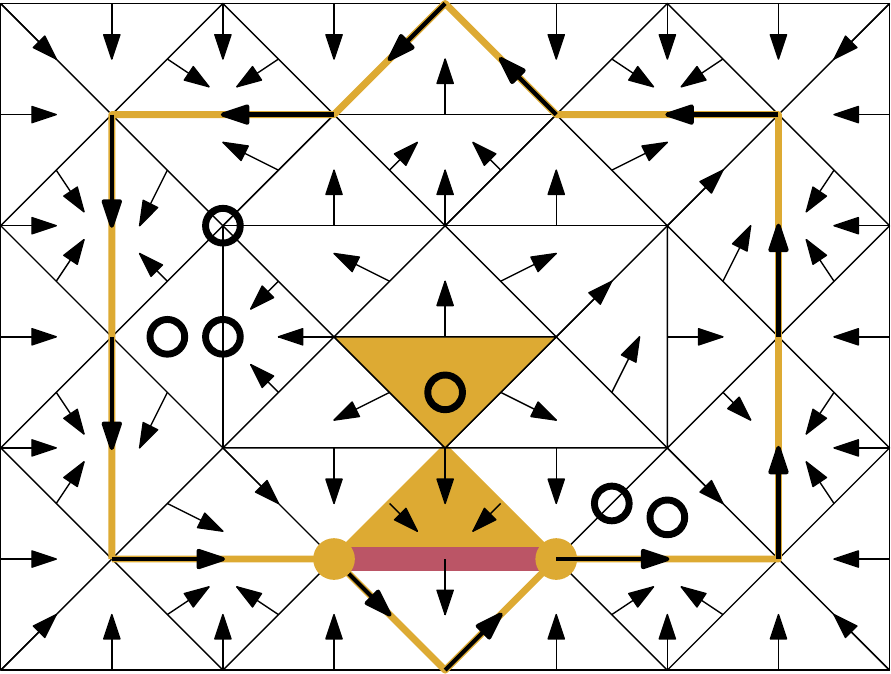}\\
\end{tabular}
\caption{Three invariant sets on the same multivector field, marked in yellow. The invariant set on the left is isolated by the entire rectangle. The invariant set in the middle is isolated by its closure, but not by the rectangle. This is because there is a path, marked in red, that starts in the invariant set, leaves the invariant set, and then re-enters the invariant set, all while staying within the rectangle. The invariant set on the right is isolated by neither its closure nor the rectangle, because there is a path from a yellow triangle, to the red edge, to a yellow vertex.}
\label{fig:isolated-invariant-sets}
\end{figure} 

\begin{figure}[htbp]
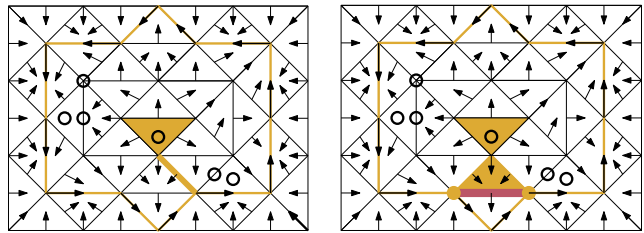

\centering
\begin{tabular}{cc}
  \includegraphics[width=40mm]{fig/invariant-set-fourth-example.pdf}&
    \includegraphics[width=40mm]{fig/invariant-set-fifth-example.pdf}\\
\end{tabular}
\caption{Two invariant sets, marked in yellow, over the same multivector field. On the right, the invariant set includes the two yellow vertices marked with filled discs, but it excludes the red edge. The invariant set on the left is not $\cV$-compatible, while the invariant set on the right is. }
\label{fig:v-compatible-invariant-sets}
\end{figure} 

The \emph{invariant part} of a set $A \subseteq K$, denoted $\inv_{\cV}(A)$, is given by the set of simplices $\sigma \in A$ for which there exists an essential solution $\rho \; : \; \mathbb{Z} \to A$ where $\rho(i) = \sigma$ for some $i \in \mathbb{Z}$. A set $S \subseteq K$ is an \emph{invariant set} if and only if $S = \inv_{\cV}(S)$. For examples of invariant sets, see Figure \ref{fig:invariant-sets}. We are interested in a special type of invariant set. 
\begin{definition}[Isolated Invariant Set, Isolating Set]
Let $S$ be an invariant set under $\cV$. 
If there exists a closed set $N$ such that $F_{\cV}(S)\subseteq N$ and every path $\rho \; : \; \mathbb{Z}_{[0,n]} \to N$ where $\rho(0), \rho(n) \in S$ has the property that $\rho(\mathbb{Z}_{[0,n]}) \subseteq S$, then $S$ is \emph{isolated} by $N$ and $S$ is an \emph{isolated invariant set}.
Moreover, the set $N$ is an \emph{isolating set} for $S$.
\end{definition}
Figure \ref{fig:isolated-invariant-sets} illustrates the concept of isolation. An invariant set $S$ is \emph{$\cV$-compatible} if $S$ is equal to the union of a set of multivectors in $\cV$. For examples, see Figure \ref{fig:v-compatible-invariant-sets}. This gives an equivalent formulation of an isolated invariant set. 
\begin{proposition}\cite[Proposition 4.1.21]{lipinski_phd}
  An invariant set $S$ is isolated if and only if it is convex and $\cV$-compatible.
  \label{prop:isolated_equiv_convex_vcomp}
\end{proposition}

%% file: 3trackingprotocol.tex
\section{Tracking Isolated Invariant Sets}

In this section, we introduce the protocol for tracking an isolated invariant set across multivector fields. Results in the continuous theory imply that under a sufficiently small perturbation, some homological features of an isolated invariant set do not change. Hence, we require a notion of a small perturbation of a multivector field. In particular, let $\cV$ and $\cV'$ denote two multivector fields on $K$. If each multivector $V' \in \cV'$ is contained in a multivector $V \in \cV$, $|\cV \setminus \cV'| = 1$, and $| \cV' \setminus \cV| = 2$, then $\cV'$ is an \emph{atomic refinement} of $\cV$. It is so-called because $\cV'$ is obtained by ``splitting'' exactly one multivector in $\cV$ into two multivectors, while all the other multivectors remain the same. Symmetrically, we say that $\cV$ is an \emph{atomic coarsening} of $\cV'$. More broadly, it is said that $\cV$ and $\cV'$ are \emph{atomic rearrangements} of each other. In Figures \ref{fig:refinement}, \ref{fig:coarsen-within}, \ref{fig:coarsen-outside}, \ref{fig:coarsen-between}, and \ref{fig:persistence-step}, the two multivector fields are atomic rearrangements of each other. In these figures, we draw the multivectors that are splitting or merging in red.

Given an isolated invariant set $S$ under $\cV$, and an atomic rearrangement of $\cV$ denoted $\cV'$, we aim to find an isolated invariant set $S'$ that is a minimal perturbation of $S$. We accomplish this through two mechanisms: \emph{continuation} and \emph{persistence}. When we use continuation, or when we attempt to \emph{continue}, we check if there exists an $S'$ under $\cV'$ that is in some sense the same as $S$. If there is at least one such $S'$, then we choose a canonical one. This is explained in Section \ref{sec:continuation}. If there is no $S'$ to which we can continue, then we use persistence. In particular, we choose a canonical isolated invariant set $S'$ under $\cV'$, and while $S$ does not continue to $S'$, we can use zigzag persistence to observe which features of $S$ are absorbed by $S'$. We elaborate on this scheme in Section \ref{sec:persistence}. To choose $S'$, we require the following result.

\begin{proposition}\cite[Corollary 4.1.22]{lipinski_phd}
    Let $A$ be a convex and $\cV$-compatible set. Then $\inv_\cV(A)$ is an isolated invariant set.
    \label{prop:invariant_part_of_convex_vcomp}
\end{proposition}

The set $S$ is an isolated invariant set by assumption, so Proposition \ref{prop:isolated_equiv_convex_vcomp} implies that $S$ is convex and $\cV$-compatible. Thus, if $S$ is also $\cV'$-compatible, a natural choice is then to use Proposition \ref{prop:invariant_part_of_convex_vcomp} and take $S' := \inv_{\cV'}(S)$. However, if $S$ is not $\cV'$-compatible, then the situation is more complicated. The set $S$ is not $\cV'$-compatible precisely when $\cV'$ is an atomic coarsening of $\cV$, and the unique multivector $V \in \cV' \setminus \cV$, occasionally called the \emph{merged multivector}, has the properties that $V \cap S \neq \emptyset$ and $V \not\subseteq S$. In such a case, we use the notation $\langle S \cup V \rangle_{\cV'}$ to denote the intersection of all $\cV'$-compatible and convex sets that contain $S \cup V$. The simplicial complex $K$ is $\cV'$-compatible and convex, so $\langle S \cup V \rangle_{\cV'}$ always exists and $S \subsetneq \langle S \cup V \rangle_{\cV'}$. 
We observe that $\langle S \cup V \rangle_{\cV'}$ is $\cV'$-compatible and convex, and thus it is the minimal convex and $\cV'$-compatible set that contains $S$. 
In such a case, we use Proposition \ref{prop:invariant_part_of_convex_vcomp} and take $S' := \inv_{\cV'}( \langle S \cup V \rangle_{\cV'}$ ). These principles are enumerated in the following Tracking Protocol.\\

\noindent
{\bf Tracking Protocol}\\
Given a nonempty isolated invariant set $S$ under $\cV$, and an atomic rearrangement of $\cV$ denoted $\cV'$, use the following rules to find an isolated invariant set $S'$ under $\cV'$ that corresponds to $S$.
\begin{enumerate}
    \item{Attempt to track via continuation: \label{protocol-step:continuation} }
    \begin{enumerate}
        \item If $\cV'$ is an atomic refinement of $\cV$, then take $S' := \inv_{\cV'}(S)$. \label{protocol-step:refinement}
        \item If $\cV'$ is an atomic coarsening of $\cV$, and the unique merged multivector $V$ has the property that $V \subseteq S$, then take $S' := \inv_{\cV'}(S)$. \label{protocol-step:coarsen-within}
        \item If $\cV'$ is an atomic coarsening of $\cV$, and the unique merged multivector $V$ has the property that $V \cap S = \emptyset$, then take $S' := \inv_{\cV'}(S) = S$.\label{protocol-step:coarsen-outside}
        \item If $\cV'$ is an atomic coarsening of $\cV$ and the unique merged multivector $V$ satisfies the formulae $V \cap S \neq \emptyset$ and $V \not\subseteq S$, then consider $A = \langle S \cup V \rangle_{\cV'}$. If $\inv_{\cV}(A) = S$, then take $S' := \inv_{\cV'}(A)$. \label{protocol-step:coarsen-between}
        \item Else, it is impossible to track via continuation. \label{protocol-step:continuation-impossible}
    \end{enumerate}
    \item {If it is impossible to track via continuation, then attempt to track via persistence:} \label{protocol-step:general-persistence}
    \begin{enumerate}
        \setcounter{enumii}{5}
        \item If $A := \langle S \cup V \rangle_{\cV}$, then take $S' := \inv_{\cV'}(A)$. If $S$ and $S'$ are isolated by a common isolating set, then use the technique in Equation \ref{eqn:general-cont-filtration} to find a zigzag filtration connecting them. \label{protocol-step:persistence} 
        \item Otherwise, there is no natural choice of $S'$. See Appendix \ref{sec:strategy} for a possible strategy. \label{protocol-step:impossible}
    \end{enumerate}
\end{enumerate}

We include an example of Step \ref{protocol-step:refinement} in Figure \ref{fig:refinement}, Step \ref{protocol-step:coarsen-within} in Figure \ref{fig:coarsen-within}, Step \ref{protocol-step:coarsen-outside} in Figure \ref{fig:coarsen-outside}, Step \ref{protocol-step:coarsen-between} in Figure \ref{fig:coarsen-between}, and Step \ref{protocol-step:persistence} in Figure \ref{fig:persistence-step}. Each figure depicts a multivector field and a seed isolated invariant set on the left, and an atomic rearrangement and the resulting isolated invariant set on the right. By iteratively applying this protocol (until $S' = \emptyset$, in which case we are done), we can track how an isolated invariant set changes across several atomic rearrangements. See Figure \ref{fig:iterative-tracking} and the associated barcode in Figure \ref{fig:iterative-tracking-barcode}. The following proposition shows that any two multivector fields $\cV_1$ and $\cV_2$ can be related by a sequence of atomic rearrangements, and hence the Tracking Protocol can be used to track how an isolated invariant set changes across an arbitrary sequence of multivector fields.

\begin{proposition}
For each pair of multivector fields $\cV$ and $\cV'$ over $K$, there exists a sequence $\cV = \cV_1$, $\cV_2$, \ldots, $\cV_n = \cV'$ where each each $\cV_i$ is an atomic rearrangement of $\cV_{i-1}$ for $i > 1$. 
\end{proposition}
\begin{proof}
    Let $\mathcal{W}$ denote the multivector field over $K$ where each simplex $\sigma \in K$ is contained in its own multivector. That is, $\mathcal{W} := \{ \{ \sigma \} \; | \; \sigma \in K \}$. We will show that there exists a sequence $\cV = \cV_1, \cV_2, \ldots, \cV_n = \mathcal{W}$ and a sequence $\cV' = \cV'_1, \cV'_2, \ldots, \cV'_m = \mathcal{W}$ where each $\cV_{i+1}$ (or $\cV'_{i+1}$) is an atomic refinement of $\cV_i$ (or $\cV'_{i}$). This immediately gives a sequence $\cV = \cV_1, \cV_2, \ldots, \cV_n = \mathcal{W} = \cV'_{m}, \ldots, \cV'_{1} = \cV'$. Note that because the reverse of an atomic refinement is an atomic coarsening, this sequence of multivectors is a sequence of $n-1$ consecutive atomic refinements followed by $m-1$ atomic coarsenings. 
    
    Hence, all that is necessary is to find a sequence $\cV = \cV_1, \cV_2, \ldots, \cV_n = \mathcal{W}$ and $\cV' = \cV'_1, \cV'_2, \ldots, \cV'_m = \mathcal{W}$. Without loss of generality, we consider the former. Consider an arbitrary multivector $V \in \cV$ where $|V| > 1$ (if there is no such $V$, then $\cV = \mathcal{W}$ and we are done). Take a maximal element $\sigma \in V$ (with respect to $\leq$), and let $\cV_1 := (\cV \setminus \{V\}) \cup \{ \{ \sigma \}\} \cup \{V \setminus \{\sigma\}\}$. Effectively, we have partitioned the multivector $V$ into one multivector that only contains $\sigma$ and the remainder of the multivector. If we iteratively repeat this process, then we must arrive at $\mathcal{W}$ in a finite number of steps, because each step creates a new multivector consisting of a single simplex, and multivectors are never merged. Furthermore, note that each time we split a multivector constitutes an atomic refinement. Hence, by iterating this process, we obtain a sequence $\cV = \cV_1, \cV_2, \ldots, \cV_n = \mathcal{W}$ where each multivector field is related to its predecessor by atomic refinement. Thus, we can combine these sequences as previously described, and we are done. 
\end{proof}

\begin{figure}[htbp]
\centering
\begin{tabular}{cc}
  \includegraphics[height=40mm]{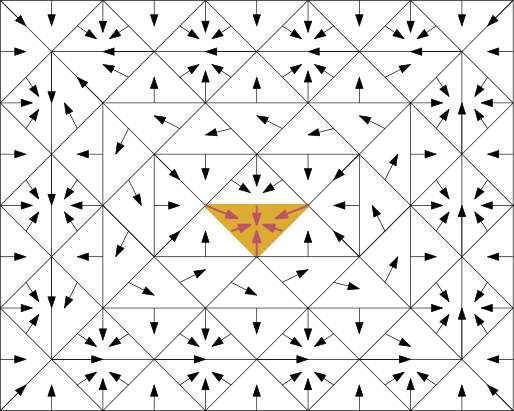}&
  \includegraphics[height=40mm]{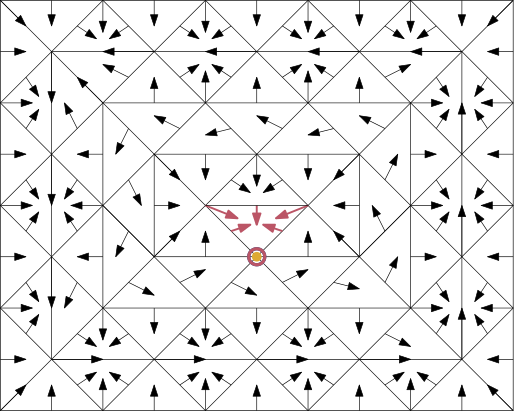}\\
\end{tabular}
\caption{Applying Step \ref{protocol-step:refinement} to an invariant set (yellow, left) to get a new one (yellow, right). }
\label{fig:refinement}
\end{figure} 

\begin{figure}[htbp]
\centering
\begin{tabular}{cc}
  \includegraphics[height=40mm]{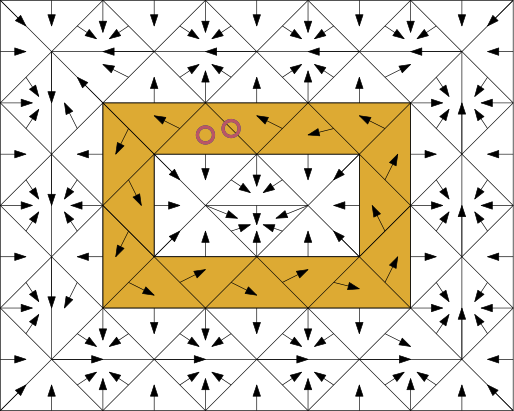}&
  \includegraphics[height=40mm]{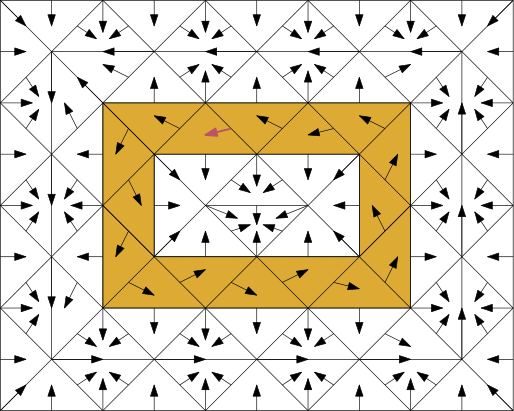}\\
\end{tabular}
\caption{Applying Step \ref{protocol-step:coarsen-within} to an invariant set (yellow, left) to get a new one (yellow, right). }
\label{fig:coarsen-within}
\end{figure} 

\begin{figure}[htbp]
\centering
\begin{tabular}{cc}
  \includegraphics[height=40mm]{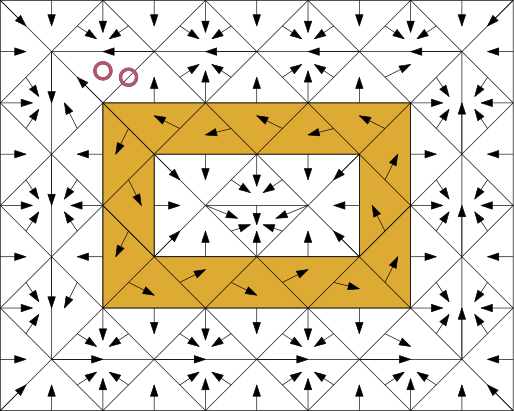}&
  \includegraphics[height=40mm]{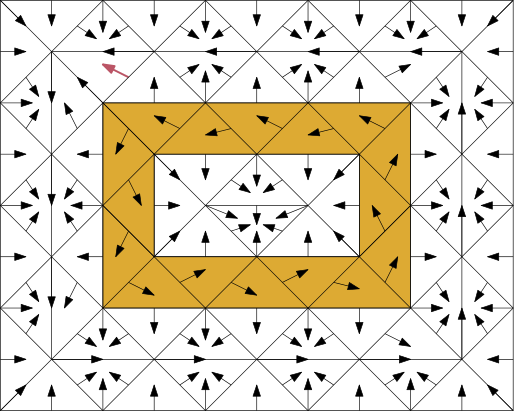}\\
\end{tabular}
\caption{Applying Step \ref{protocol-step:coarsen-outside} to an invariant set (yellow, left) to get a new one (yellow, right). The merged vector is outside of the invariant set on the left, so the invariant sets are the same.}
\label{fig:coarsen-outside}
\end{figure} 

\begin{figure}[htbp]
\centering
\begin{tabular}{cc}
  \includegraphics[height=40mm]{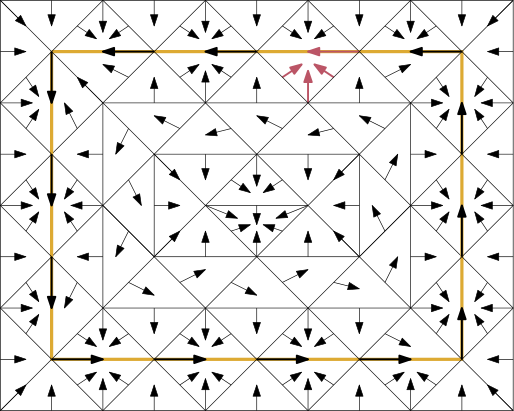}&
  \includegraphics[height=40mm]{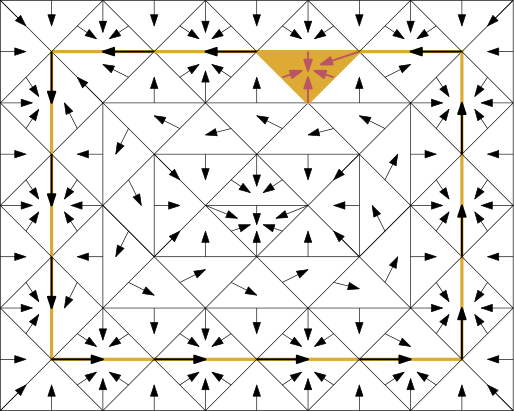}\\
\end{tabular}
\caption{Applying Step \ref{protocol-step:coarsen-between} to an invariant set (yellow, left) to get a new one (yellow, right). }
\label{fig:coarsen-between}
\end{figure} 

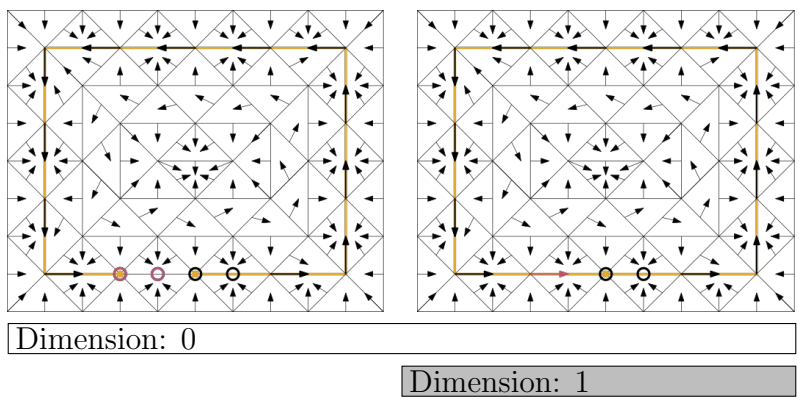
\begin{figure}[htbp]
\centering
\begin{tabular}{cc}
  \includegraphics[height=40mm]{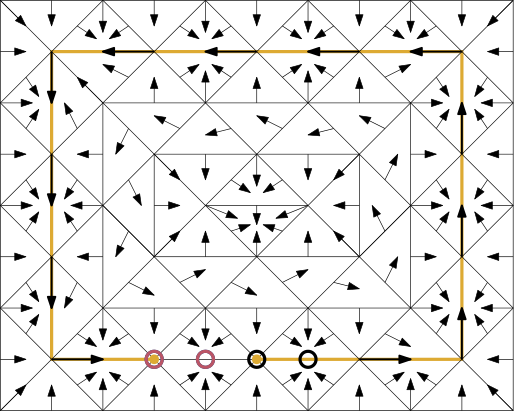}&
  \includegraphics[height=40mm]{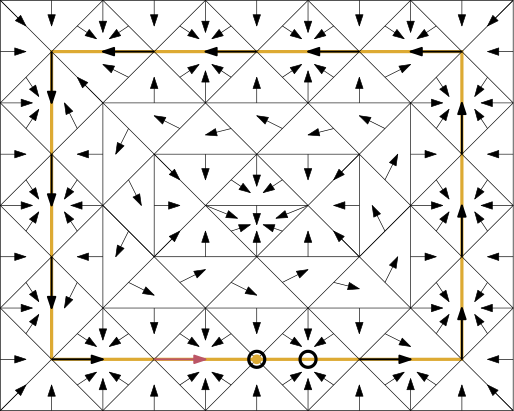}\\
  \multicolumn{2} {l} {
  \begin{tikzpicture}[outer sep = 0, inner sep = 0]
  \draw[fill=white,outer sep=0, inner sep=0] (0.0,-0.3) -- (10.47,-0.3) -- (10.47,0.1) -- (0.0,0.1) -- cycle;
  \node[align=left,color=black] at (1.3,-0.1) {Dimension: 0};
  \end{tikzpicture} }\\
  \multicolumn{2} {r} {
  \begin{tikzpicture}[outer sep = 0, inner sep = 0]
  \draw[fill=light-gray,outer sep=0, inner sep=0] (0.0,-0.3) -- (5.235,-0.3) -- (5.235,0.1) -- (0.0,0.1) -- cycle;
  \node[align=left,color=black] at (1.3,-0.1) {Dimension: 1};
  \end{tikzpicture} }
\end{tabular}
\caption{Applying Step \ref{protocol-step:persistence} to an invariant set (yellow, left) to get a new one (yellow, right). The associated persistence barcode is depicted below the figures.}
\label{fig:persistence-step}
\end{figure} 

\begin{figure}
    \centering
    \begin{subfigure}{.3\linewidth}
        \centering
        \includegraphics[width=42mm]{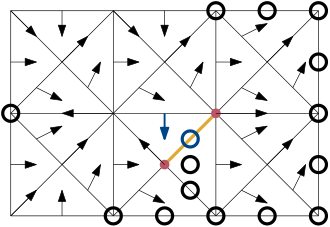}
        \caption{Initial multivector field}
        \label{fig:tracking-first}
        \end{subfigure}
    \hfill
    \begin{subfigure}{.3\linewidth}
        \centering
        \includegraphics[width=42mm]{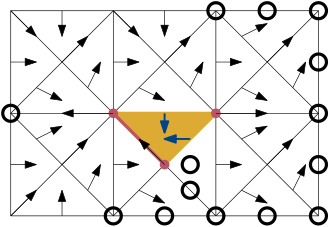}
        \captionsetup{textformat=simple}
        \caption{Atomic coarsening of \ref{fig:tracking-first}}
        \label{fig:tracking-second}
    \end{subfigure}
    \hfill
    \begin{subfigure}{.3\linewidth}
        \centering
        \includegraphics[width=42mm]{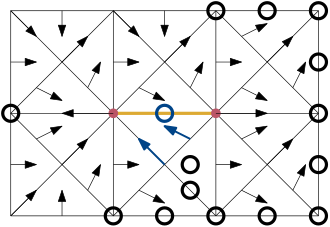}
        \captionsetup{textformat=simple}
        \caption{Atomic refinement of \ref{fig:tracking-second}}
        \label{fig:tracking-third}
    \end{subfigure}
    
    \bigskip
    
    \begin{subfigure}{.3\linewidth}
        \centering
        \includegraphics[width=42mm]{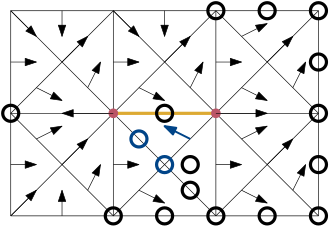}
        \caption{Atomic refinement of \ref{fig:tracking-third}}
        \label{fig:tracking-fourth}
        \end{subfigure}
    \hfill
    \begin{subfigure}{.3\linewidth}
        \centering
        \includegraphics[width=42mm]{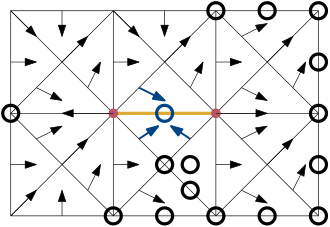}
        \captionsetup{textformat=simple}
        \caption{Atomic coarsening of \ref{fig:tracking-fourth}}
        \label{fig:tracking-fifth}
    \end{subfigure}
    \hfill
    \begin{subfigure}{.3\linewidth}
        \centering
        \includegraphics[width=42mm]{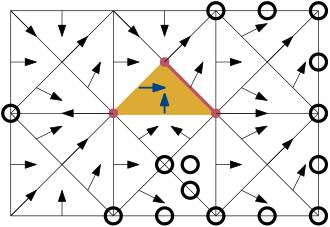}
        \captionsetup{textformat=simple}
        \caption{Atomic coarsening of \ref{fig:tracking-fifth}}
        \label{fig:tracking-sixth}
    \end{subfigure}
    
    \bigskip
    
    \begin{subfigure}{.3\linewidth}
        \centering
        \includegraphics[width=42mm]{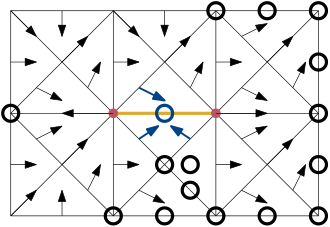}
        \caption{Atomic refinement of \ref{fig:tracking-sixth}}
        \label{fig:tracking-seventh}
        \end{subfigure}
    \hfill
    \begin{subfigure}{.3\linewidth}
        \centering
        \includegraphics[width=42mm]{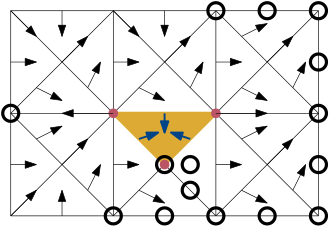}
        \captionsetup{textformat=simple}
        \caption{Atomic coarsening of \ref{fig:tracking-seventh}}
        \label{fig:tracking-eighth}
    \end{subfigure}
    \hfill
    \begin{subfigure}{.3\linewidth}
        \centering
        \includegraphics[width=42mm]{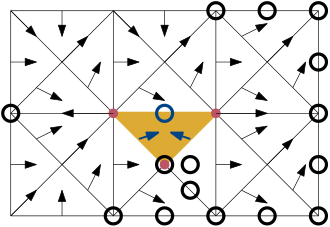}
        \captionsetup{textformat=simple}
        \caption{Atomic refinement of \ref{fig:tracking-eighth}}
        \label{fig:tracking-ninth}
    \end{subfigure}
    
    \caption{Subfigure \ref{fig:tracking-first} contains an initial multivector field and a seed isolated invariant set, which is a yellow edge. Each subsequent subfigure contains a multivector field that is an atomic refinement or atomic coarsening of the previous. The isolated invariant set that we get by iteratively applying the Tracking Protocol is depicted in yellow. Splitting and merging multivectors are in blue.}
    \label{fig:iterative-tracking}

\end{figure}

\begin{figure}
    \centering

    \begin{tikzpicture}[outer sep = 0, inner sep = 0]
    \draw[fill=light-gray,outer sep=0, inner sep=0] (-5.0,-0.3) -- (9,-0.3) -- (9,0.1) -- (-5.0,0.1) -- cycle;
    \node[align=left,color=black] at (-3.7,-0.1) {Dimension: 1};
    \draw[fill=light-gray,outer sep=0, inner sep=0] (6.0,-0.8) -- (9,-0.8) -- (9,-0.4) -- (6.0,-0.4) -- cycle;
    \node[align=left,color=black] at (7.3,-0.6) {Dimension: 1};
    \end{tikzpicture}

    \caption{The barcode associated with the tracked invariant sets in Figure \ref{fig:iterative-tracking}. Starting with subfigure \ref{fig:tracking-eighth}, we see the birth of a new $1$-dimensional homology generator.}
    \label{fig:iterative-tracking-barcode}

\end{figure}
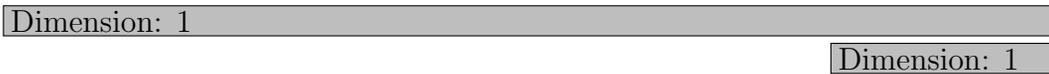

%% file: 4continuation.tex
\section{Tracking via Continuation}\label{sec:continuation}

Now, we introduce continuation in the combinatorial setting, and we justify the canonicity of the choices made in Step \ref{protocol-step:continuation} of the Tracking Protocol. In addition, we show that if Step \ref{protocol-step:continuation} is used to obtain $S'$ from $S$, then $S$ and $S'$ are related by continuation. Continuation is closely related to the Conley index, so we begin with a brief review of the topic.

\subsection{Index Pairs and the Conley Index}

Originally developed in the classical setting by Conley \cite{Co78}, the Conley index associates a homological invariant with each isolated invariant set. It is defined through index pairs.

\begin{definition}[Index Pair]
Let $S$ denote an isolated invariant set under $\cV$, and let $P$ and $E$ denote closed sets where $E \subseteq P$. If the following all hold, then $(P,E)$ is an \emph{index pair} for $S$: 
\begin{enumerate}
    \item\label{it:index_pair_i} $F_{\cV}(P \setminus E) \subseteq P$
    \item\label{it:index_pair_ii} $F_{\cV}(E) \cap P \subseteq E$
    \item\label{it:index_pair_iii} $S = \inv_{\cV}(P \setminus E)$
\end{enumerate}
\label{def:index_pair}
\end{definition}

\begin{figure}[htbp]
\centering
\begin{tabular}{ccc}
  \includegraphics[height=32mm]{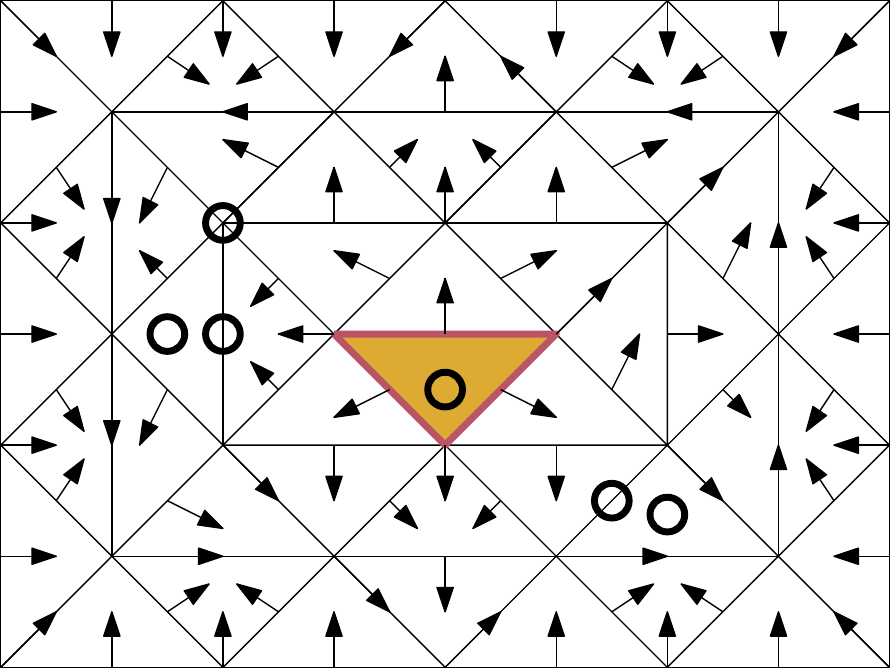}&
  \includegraphics[height=32mm]{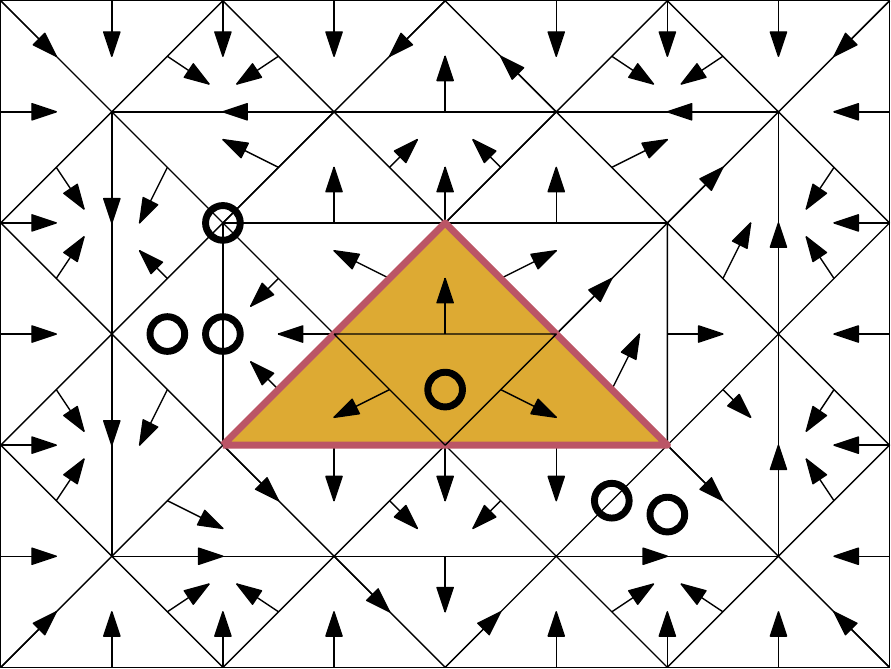}&
  \includegraphics[height=32mm]{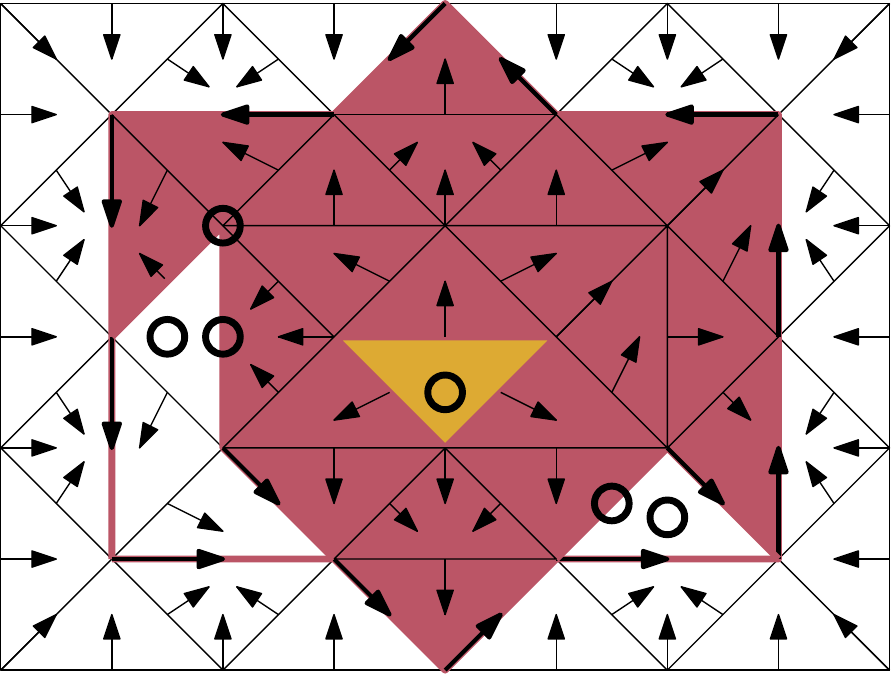}\\
\end{tabular}
\caption{
All three images depict an index pair for the yellow triangle marked with a black circle. $P$ is given by the red and yellow simplices, while $E$ is given by the red simplices.}
\label{fig:index-pair-example}
\end{figure} 

For examples, see Figure \ref{fig:index-pair-example}. If $(P,E)$ is an index pair for $S$, then the $k$-dimensional Conley index of $S$ is $H_k(P,E)$. 
The authors in \cite{LKMW2020} showed that the Conley index is well-defined.
\begin{theorem}{\cite[Theorem 5.16]{LKMW2020}}
Let $(P,E)$ and $(P',E')$ denote index pairs for $S$. Then $H_k(P,E) = H_k(P',E')$ for all $k \geq 0$. 
\end{theorem}

For a single isolated invariant set, there may be many possible index pairs. 
However, we can choose a canonical one, namely, the minimal index pair.
\begin{proposition}{\cite[Proposition 5.3]{LKMW2020}} 
Let $S$ denote an isolated invariant set. The pair $(\cl(S),\mo(S))$ is an index pair for $S$.
\label{prop:closure-mouth}
\end{proposition}

The following two propositions show that convex and $\cV$-compatible sets are crucial for finding index pairs.

\begin{proposition}\cite[Proposition 5.6]{LKMW2020} 
    Let $(P,E)$ be an index pair under $\cV$. Then $P\setminus E$ is convex and $\cV$-compatible.
    \label{prop:P_minus_E_convex_vcomp}
\end{proposition}

\begin{proposition}
If $A$ is convex and $\cV$-compatible, then $(\cl(A),\mo(A))$ is an index pair for $\inv_{\cV}(A)$.
\label{prop:convex-and-v-compatible-gives-index-pair}
\end{proposition}
\begin{proof}
    By Proposition \ref{prop:invariant_part_of_convex_vcomp} the set $S=\inv_{\cV}(A)$ is an isolated invariant set.
    Since $\cl (A) \setminus \mo (A)= A$, we immediately get condition \ref{it:index_pair_iii} from Definition $\ref{def:index_pair}$.
    Since $A$ is $\cV$-compatible we get $F_\cV(A)=\cl A$, and thus, condition \ref{it:index_pair_i}.
    To see condition \ref{it:index_pair_ii} consider $x\in F_\cV(\mo (A))$.
    By the definition of $F_\cV$ there exists an $a \in \mo(A)$ such that either $x\in[a]_{\cV}$ or $x\in\cl(a)$.
    In the first case $x \not\in A$, because $A$ is $\cV$-compatible and $ a \not\in A$. Therefore $[a]_{\cV}\cap\cl(A)\subseteq \cl(A)\setminus A=\mo(A)$.
    If $x\in\cl(a)$ then $x\in\mo A$, because $\mo(A)$ is closed.
    Hence, it follows that $F_\cV(\mo (A))\cap \cl(A) \subseteq\mo(A)$.
\end{proof}

\subsection{Combinatorial Continuation and the Tracking Protocol}

We now move to placing continuation in the combinatorial setting and explaining Step \ref{protocol-step:continuation} of the Tracking Protocol. In essence, a continuation captures the presence of the ``same'' isolated invariant set across multiple multivector fields. We then show that Step \ref{protocol-step:continuation} of the Tracking Protocol does use continuation to track an isolated invariant set. 
\begin{definition}
Let $S_1$, $S_2$, \ldots, $S_n$ denote a sequence of isolated invariant sets under the multivector fields $\cV_1$, $\cV_2$, \ldots, $\cV_n$, where each $\cV_i$ is defined on a fixed simplicial complex $K$. 
We say that isolated invariant set $S_1$ \emph{continues} to isolated invariant set $S_n$ whenever there exists a sequence of index pairs $(P_1, E_1)$, $(P_2,E_2)$, \ldots, $(P_{n-1},E_{n-1})$ where $(P_i,E_i)$ is an index pair for both $S_i$ and $S_{i+1}$. Such a sequence is a \emph{sequence of connecting index pairs}.
\end{definition}

\begin{figure}[htbp]
\centering
\begin{tabular}{cc}
  \includegraphics[height=40mm]{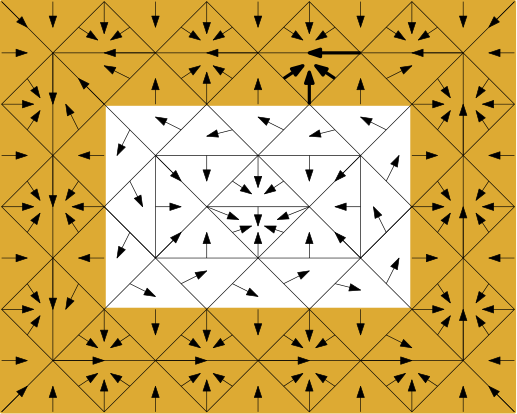}&
  \includegraphics[height=40mm]{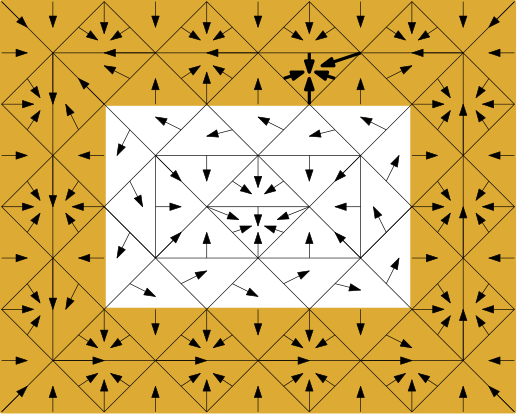}\\
\end{tabular}
\caption{An index pair, where $P$ is in yellow and $E$ is empty, for the isolated invariant sets in Figure \ref{fig:coarsen-between}. There is a common index pair for both isolated invariant sets, so they form a continuation. }
\label{fig:continuation-ex-attractor}
\end{figure} 

\begin{figure}[htbp]
\centering
\begin{tabular}{cc}
  \includegraphics[height=40mm]{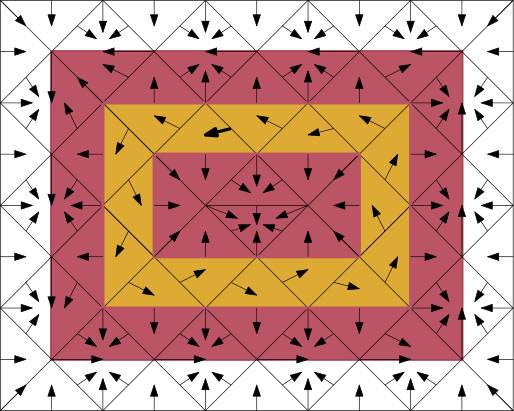}&
  \includegraphics[height=40mm]{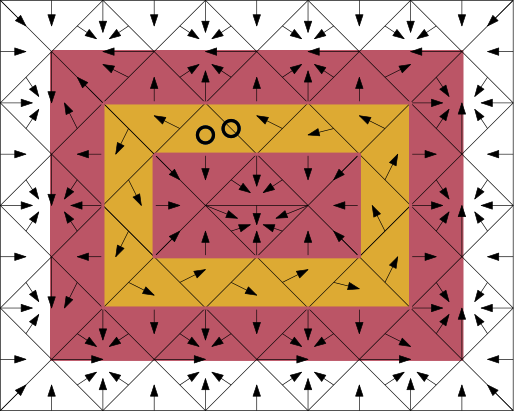}\\
\end{tabular}
\caption{An index pair, where $P$ is given by the yellow and red simplices and $E$ is given by the red simplices, for the isolated invariant sets in Figure \ref{fig:coarsen-within}. Thus, they form a continuation. }
\label{fig:continuation-ex-repeller}
\end{figure} 

Each index pair $(P_i,E_i)$ in a connecting sequence of index pairs is an index pair for a pair of consecutive isolated invariant sets $S_i$ and $S_{i+1}$ (see Figures \ref{fig:continuation-ex-attractor} and \ref{fig:continuation-ex-repeller}). Hence, the isolated invariant sets in the continuation all have the same Conley index. In this sense, we are capturing the ``same'' isolated invariant set. In Step \ref{protocol-step:continuation} of the Tracking Protocol, we first attempt to track the isolated invariant set $S$ via continuation. That is, if we use Step \ref{protocol-step:continuation}, then we choose $S'$ such that $S$ and $S'$ have a common index pair, say $(P,E)$. It so happens that $(P,E)$ is easy to find algorithmically. We begin with the refinement case, or Step \ref{protocol-step:refinement}. 
\begin{theorem}
Let $\cV$ and $\cV'$ denote multivector fields where $\cV'$ is an atomic refinement of $\cV$. Let $A$ be a $\cV$-compatible and convex set. The pair $(\cl(A), \mo(A))$ is an index pair for both $\inv_{\cV}(A)$ under $\cV$ and $\inv_{\cV'}(A)$ under $\cV'$.
\label{thm:refinement-case}
\end{theorem}
\begin{proof}
    By Proposition \ref{prop:convex-and-v-compatible-gives-index-pair}, if $A$ is convex and $\cV$-compatible, then $(\cl(A), \mo(A))$ is an index pair for $\inv_{\cV}(A)$. By assumption, $A$ is $\cV$-compatible, so $(\cl(A),\mo(A))$ is an index pair for $\inv_{\cV}(A)$. Hence, if we can show that $A$ is $\cV'$-compatible, then it will immediately follow by Proposition \ref{prop:convex-and-v-compatible-gives-index-pair} that $(\cl(A),\mo(A))$ is an index pair for $\inv_{\cV'}(A)$. Because $A$ is $\cV$-compatible, it follows that there exists a set of multivectors $R \subseteq \cV$ where $A = \cup_{V \in R} V$. Recall that as $\cV'$ is an atomic refinement of $\cV$, there is exactly one multivector $W \in \cV \setminus \cV'$. If $W \not\in R$, then we are done, and $A$ is necessarily $\cV'$-compatible, as each multivector in $R$ is a multivector in $\cV'$. If $W \in R$, then we observe that there exist two multivectors $W_1,W_2 \in \cV' \setminus \cV$ where $W = W_1 \cup W_2$. In such a case, it follows easily that if $R' = (R \setminus \{W\}) \cup \{ W_1, W_2\}$, then each multivector in $R'$ is a multivector in $\cV'$ and $A = \cup_{V \in R'} V$. Hence, $A$ is $\cV'$-compatible, and $(\cl(A),\mo(A))$ is an index pair for $\inv_{\cV'}(A)$. 
\end{proof}

In Step \ref{protocol-step:refinement} of the Tracking Protocol, where $\cV'$ is an atomic refinement of $\cV$, we choose $S' := \inv_{\cV'}(S)$. By Proposition \ref{prop:isolated_equiv_convex_vcomp}, it follows that $S$ is $\cV$-compatible. By identical reasoning to that presented in the proof of Theorem \ref{thm:refinement-case}, it follows that $S$ is also $\cV'$-compatible. Hence, Theorem \ref{thm:refinement-case} implies that $(\cl(S),\mo(S))$ is an index pair for both $S = \inv_{\cV}(S)$ and $S' = \inv_{\cV'}(S)$. Thus, $S$ and $S'$ share an index pair.  

The case of an atomic coarsening, corresponding to Steps \ref{protocol-step:coarsen-within}, \ref{protocol-step:coarsen-outside}, and \ref{protocol-step:coarsen-between} of the Tracking Protocol, is more complicated. Recall that if $\cV'$ is an atomic coarsening of $\cV$, then the unique multivector $V \in \cV' \setminus \cV$ is called the \emph{merged multivector}. 
\begin{theorem}
Let $\cV$ and $\cV'$ denote multivector fields where $\cV'$  is an atomic coarsening of $\cV$. 
Let $A$ be a convex and $\cV$-compatible set, and let $V \in \cV'$ be the unique merged multivector. 
If $V \subseteq A$ or $V \cap A = \emptyset$, then $(\cl(A), \mo(A))$ is an index pair for both $\inv_{\cV}(A)$ and $\inv_{\cV'}(A)$. 
\label{thm:coarsening-case-1-2}
\end{theorem}
\begin{proof}
    If $V \cap A = \emptyset$, then $A$ is both $\cV$-compatible and $\cV'$-compatible. Thus, Proposition \ref{prop:convex-and-v-compatible-gives-index-pair} implies that $(\cl(A),\mo(A))$ is an index pair for both $S = \inv_{\cV}(A)$ and $S' = \inv_{\cV'}(A)$. 
    
    If $V \subseteq A$, then by the same reasoning as in the proof of Theorem \ref{thm:refinement-case}, it follows that $A$ is both $\cV$-compatible and $\cV'$-compatible. Thus, Proposition \ref{prop:convex-and-v-compatible-gives-index-pair} implies that $(\cl(A),\mo(A))$ is an index pair for both $\inv_{\cV}(A)$ and $\inv_{\cV}(A)$. 
\end{proof}


By Proposition \ref{prop:isolated_equiv_convex_vcomp}, $S$ is convex and $\cV$-compatible. Theorem \ref{thm:coarsening-case-1-2} implies that if $V \subseteq S$ or $V \cap S = \emptyset$, then $(\cl(S),\mo(S))$ is an index pair for both $\inv_{\cV}(S) = S$ and $\inv_{\cV'}(S) = S'$. In Steps \ref{protocol-step:coarsen-within} and \ref{protocol-step:coarsen-outside} of the Tracking Protocol, $S'$ is chosen as $\inv_{\cV'}(S)$. Hence, the index pair $(\cl(S),\mo(S))$ is an index pair for both $S$ and $S'$. 

A more complicated case is Step \ref{protocol-step:coarsen-between}, where $V \cap S \neq \emptyset$ and $V \not\subseteq S$. Recall that $A := \langle S \cup V \rangle_{\cV'}$ denotes the intersection of all convex and $\cV'$-compatible sets that contain $S \cup V$, and in particular, $A$ is convex and $\cV'$-compatible. In Step \ref{protocol-step:coarsen-between} of the Tracking Protocol, we first check if $S = \inv_{\cV}( A )$. By Proposition \ref{prop:convex-and-v-compatible-gives-index-pair}, if $S = \inv(A)$, then $(\cl(A),\mo(A))$ is an index pair for $S$. The set $\langle S \cup V \rangle_{\cV'}$ is necessarily $\cV$-compatible, because it is $\cV'$-compatible by construction and it contains the unique merged multivector. Hence, Proposition \ref{prop:invariant_part_of_convex_vcomp} implies that $S' := \inv_{\cV'}(A)$ is an isolated invariant set. Thus, Proposition \ref{prop:convex-and-v-compatible-gives-index-pair} implies that $(\cl(A),\mo(A))$ is also an index pair for $S'$. Hence, if Step \ref{protocol-step:coarsen-between} gives $S'$, there is an index pair for $S$ and $S'$. 

In Step \ref{protocol-step:continuation-impossible} of the Tracking Protocol, we claim that if $S \neq \inv_{\cV}(A)$, then it is not possible to continue. Equivalently, there is no $S'$ that shares an index pair with $S$. 
\begin{theorem}
    Let $S$ denote an isolated invariant set under $\cV$ and let $\cV'$ denote an atomic coarsening of $\cV$ where the unique merged multivector $V \in \cV' \setminus \cV$ satisfies the formulae $V \cap S \neq \emptyset$ and $V \not\subseteq S$. 
    Furthermore, let $A := \langle S \cup V \rangle_{\cV'}$. 
    If $S \neq \inv_{\cV}(A)$, then there does not exist an isolated invariant set $S'$ under $\cV'$ for which there is an index pair $(P,E)$ satisfying $\inv_{\cV}(P \setminus E) = S$ and $\inv_{\cV'}(P \setminus E) = S'$.
\label{thm:impossible_continuation}
\end{theorem}
\begin{proof}
    Suppose that $S \neq\inv_{\cV}(A)$ and there exists an index pair, $(P,E)$, for both $S$ under $\cV$ and some $S'$ under $\cV'$.
    By Proposition \ref{prop:P_minus_E_convex_vcomp}, the set $P\setminus E$ must be convex and $\cV'$-compatible.
    Since $S \subseteq P\setminus E$ and $A$ is the smallest convex and $\cV'$-compatible set containing $S$, it follows that $A \subseteq P\setminus E$.
    Hence, $\inv_{\cV} (A) \subseteq \inv_{\cV} (P\setminus E)$. By assumption, $S \subsetneq \inv_{\cV}(A)$. Thus, $S \subsetneq \inv_{\cV}( P \setminus E)$. This implies that $(P,E)$ is not an index pair for $S$, a contradiction. 
\end{proof}

\subsection{Characterizing Tracked Isolated Invariant Sets}

Step \ref{protocol-step:continuation} of the Tracking Protocol provides an avenue for tracking an isolated invariant set across a sequence of atomic rearrangements. In this subsection, we justify the canonicity of the selected isolated invariant set in Step \ref{protocol-step:continuation} of the Tracking Protocol. First, we observe that we always have an inclusion. Theorem \ref{thm:subset-thm} follows directly from the next two results

\begin{proposition}
Let $S$ be an isolated invariant set under $\cV$, and let $S'$ denote an isolated invariant set under $\cV'$ that is obtained by applying the Tracking Protocol. 
If $S'$ is obtained via Steps \ref{protocol-step:refinement}, \ref{protocol-step:coarsen-within}, or \ref{protocol-step:coarsen-outside}, then $S' \subseteq S$. 
\end{proposition}
\begin{proof}
    In Steps \ref{protocol-step:refinement}, \ref{protocol-step:coarsen-within}, and \ref{protocol-step:coarsen-outside}, $S'$ is obtained by taking $S' := \inv_{\cV'}(S)$. By definition, $\inv_{\cV'}(S) \subseteq S$, so $S' \subseteq S$. 
\end{proof}

\begin{proposition}
Let $S$ be an isolated invariant set under $\cV$, and let $S'$ denote an isolated invariant set under $\cV'$ that is obtained via applying the Tracking Protocol. 
If $S'$ is obtained via Step \ref{protocol-step:coarsen-between} then $S \subseteq S'$ or $S'\subseteq S$.
\label{prop:isolated-invariant-set-inclusion}
\end{proposition}

\begin{proof}
    First, we claim that if $S \not\subseteq S'$, and if $V$ is the unique merged multivector $V \in \cV' \setminus \cV$, then $V \cap S' = \emptyset$. 
    
    If $S \not\subseteq S'$, then there exists a $\sigma \in S \setminus S'$. Because $\sigma \in S$, there exists an essential solution $\rho \; : \; \mathbb{Z} \to A$ under $\cV$ where $\rho(0) = \sigma$. It is easy to check that because $\cV'$ is an atomic coarsening of $\cV$, we have that for every $\tau \in K$, $F_{\cV}(\tau) \subseteq F_{\cV'}(\tau)$. Hence, $\rho$ must be a solution under $\cV'$. But, since $\sigma \in S \setminus S'$, it follows that $\rho$ is not an essential solution. 
    
    Without loss of generality, we assume that there exists a $j > 0$ such that for all $i_1, i_2 \geq j$, we have that $[\rho(i_1)]_{\cV'} = [\rho(i_2)]_{\cV'}$. Because $\rho$ is an essential solution under $\cV$, and $|\cV' \setminus \cV| = 1$, it follows that $[\rho(i_1)]_{\cV'} = [\rho(i_2)]_{\cV'} = V$. Hence, $V$ must not be critical, as if it were, then $\rho$ would be an essential solution under $\cV'$. 
    
    Now, aiming for a contradiction, assume there exists a $\tau \in V \cap S'$. Then there exists an essential solution $\rho' \; : \; \mathbb{Z} \to A$ under $\cV'$ where $\rho'(j+1) = \tau$. Thus, because $\rho(j) \in V$, we can obtain a new solution $r \; : \; \mathbb{Z} \to S'$ where $r(i) = \rho(i)$ if $i \leq j$ and $r(i) = \rho'(i)$ if $i > j$. We have shown that $\rho(j) \in V$, and by assumption, $\tau = \rho'(j+1)$ has the property that $\tau \in V$. Hence, because $\rho$ is a solution and $\rho'$ is an essential solution, we have the property that for all $k$, there exists an $i > k$ where $[ r(i) ]_{\cV'} \neq [r(k)]_{\cV'}$. We can use the same construction to guarantee that there exists an $i < k$ where $[ r(i) ]_{\cV'} \neq [r(k)]_{\cV'}$. Hence, $r$ is an essential solution under $\cV'$, where $r(0) = \sigma$. But this implies that $\sigma \in S'$, a contradiction. Hence, there can exist no such $\tau$, so $V \cap S' = \emptyset$. 
    
    Thus, we have the property that $V \cap S' = \emptyset$. Let $\rho \; : \; \mathbb{Z} \to S'$ denote an essential solution under $\cV'$. Observe that for each $i$, $[\rho(i)]_{\cV'} = [\rho(i)]_{\cV}$. Ergo, $\rho$ is also an essential solution under $\cV$. Hence, $S' \subseteq S$. 
\end{proof}.

\begin{theorem}
    If $S'$ is obtained by applying Step \ref{protocol-step:continuation} of the Tracking Protocol to $S$, then we have $S\subseteq S'$ or $S'\subseteq S$.
    \label{thm:subset-thm}
\end{theorem}

Furthermore, isolated invariant sets chosen by Step \ref{protocol-step:continuation} minimize the perturbation to $S$ in terms of the number of inclusions.

\begin{proposition}
Let $S$ be an isolated invariant set under $\cV$, and let $S'$ be an isolated invariant set under $\cV'$ that is obtained by applying Step \ref{protocol-step:continuation} of the Tracking Protocol
to $S$. 
If $S''$ is any isolated invariant set under $\cV'$ that shares a common index pair with $S$, then $S' \subseteq S''$.
Moreover, if $S''\subseteq S$, then $S'=S''$.
\label{prop:min-perturbation}
\end{proposition}
\begin{proof}
    Let $(P,E)$ be a common index pair for $S$ under $\cV$ and $S''$ under $\cV'$.
    Consider Steps \ref{protocol-step:refinement}, \ref{protocol-step:coarsen-within}, and \ref{protocol-step:coarsen-outside} where $S'=\inv_{\cV'}(S)$.
    By definition, $S\subseteq P\setminus E$, and it follows that $S'=\inv_{\cV'} S\subseteq \inv_{\cV'} (P\setminus E)= S''$.
    Moreover, if $S''\subseteq S$, we get that $S''=\inv_{\cV'} S''\subseteq \inv_{\cV'} S=S'$.
    Thus, $S'=S''$.
    
    To prove the property for Step \ref{protocol-step:coarsen-between}, notice that by Proposition \ref{prop:P_minus_E_convex_vcomp}, and the fact that $A := \langle S \cup V \rangle_{\cV'}$ is the minimal convex and $\cV'$-compatible set containing $S$, we get $A\subseteq P\setminus E$.
    Therefore $S'=\inv_{\cV'} (A)\subseteq \inv_{\cV'} (P\setminus E)=S''$.
    Similarly, if $S''\subseteq S\subseteq A$, then we get $S''=\inv_{\cV'}S''\subseteq \inv_{\cV'} A = S'$.
    Thus, $S'=S''$.
\end{proof}

%% file: 5general.tex
\section{Tracking via Persistence}
\label{sec:persistence}

In the previous section, we explicated Step \ref{protocol-step:continuation} of the protocol, which uses continuation to track an isolated invariant set across a changing multivector field. In this section, we first place continuation in the persistence framework by showing how to translate the idea of combinatorial continuation into a zigzag filtration \cite{zigzag,DW22} that does not introduce spurious information. Then, we use the persistence view of continuation to justify Step \ref{protocol-step:persistence} of the Tracking Protocol, which permits us to capture changes in an isolated invariant set when no continuation is possible. In particular, it permits us to track an isolated invariant set even in the presence of a bifurcation that changes the Conley index. If the isolated invariant set that we are tracking collides, or \emph{merges}, with another isolated invariant set, then we follow the newly formed isolated invariant set, and persistence captures which aspects of our original isolated invariant set persist into the new one. Conversely, if an isolated invariant set splits, we track the smallest isolated invariant set that contains all of the child invariant sets. We begin by reviewing some results on computing the persistence of the Conley index from \cite{DMS2020}.

\subsection{Conley Index Persistence}

In \cite{DMS2020}, the authors were interested in computing the changing Conley index across a sequence of isolated invariant sets. A naive approach to computing the persistence of the Conley index is, if given two index pairs $(P_1,E_1)$ and $(P_2,E_2)$, to take the intersection of the index pairs to obtain the zigzag filtration $(P_1, E_1) \supseteq (P_1 \cap P_2, E_1 \cap E_2) \subseteq (P_2, E_2)$. However, the intersection of index pairs is generally not an index pair, and as a consequence, the barcode associated with this zigzag filtration does not capture a changing Conley index. In addition, due to the fact that $(P_1 \cap P_2, E_1 \cap E_2)$ need not be an index pair, the barcode is frequently erratic. We include an example in Figure \ref{fig:intersection-not-index-pair}.
\begin{figure}[hbp]
\centering
\begin{tabular}{ccc}
  \includegraphics[height=32mm]{fig/small-index-pair-no-orbit.pdf}&
  \includegraphics[height=32mm]{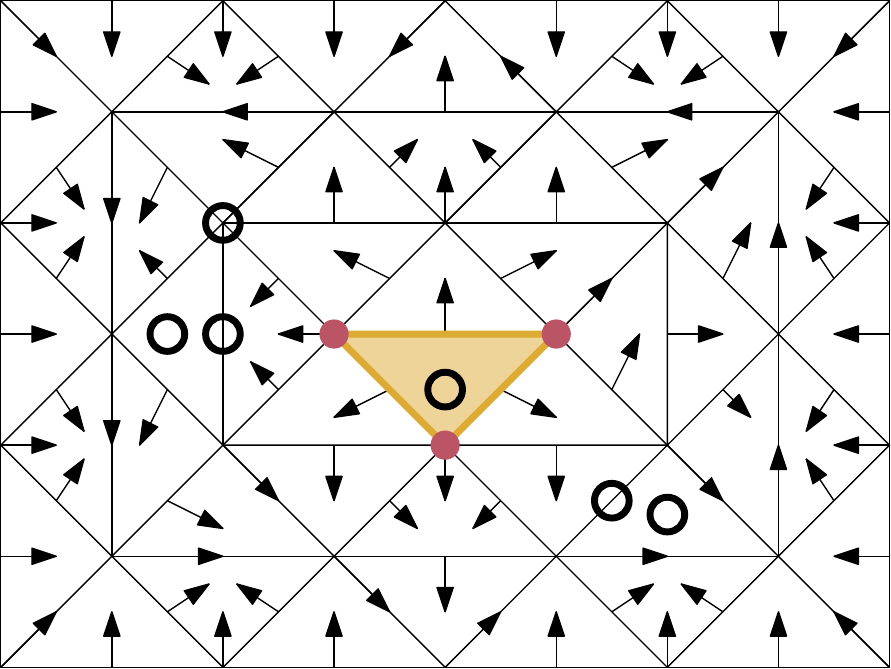}&
  \includegraphics[height=32mm]{fig/medium-index-pair-no-orbit.pdf}\\
  \multicolumn{3} {l} {
  \begin{tikzpicture}[outer sep = 0, inner sep = 0]
  \draw[fill=dark-gray,outer sep=0, inner sep=0] (0.0,-0.3) -- (4.45,-0.3) -- (4.45,0.1) -- (0.0,0.1) -- cycle;
  \node[align=left,color=white] at (1.3,-0.1) {Dimension: 2};
  \draw[fill=dark-gray,outer sep=0, inner sep=0] (9.18,-0.3) -- (13.65,-0.3) -- (13.65,0.1) -- (9.18,0.1) -- cycle;
  \node[align=left,color=white] at (10.48,-0.1) {Dimension: 2};
  \end{tikzpicture} }\\
  \multicolumn{3} {c} {
  \begin{tikzpicture}[outer sep = 0, inner sep = 0]
  \draw[fill=light-gray,outer sep=0, inner sep=0] (4.43,-0.3) -- (9.09,-0.3) -- (9.09,0.1) -- (4.43,0.1) -- cycle;
  \node[align=left,color=black] at (5.73,-0.1) {Dimension: 1};
  \end{tikzpicture} }\\
  \multicolumn{3} {c} {
  \begin{tikzpicture}[outer sep = 0, inner sep = 0]
  \draw[fill=light-gray,outer sep=0, inner sep=0] (4.43,-0.3) -- (9.09,-0.3) -- (9.09,0.1) -- (4.43,0.1) -- cycle;
  \node[align=left,color=black] at (5.73,-0.1) {Dimension: 1};
  \end{tikzpicture} }
\end{tabular}
\caption{
All three images depict the same multivector field, which includes a yellow repelling fixed point (triangle, marked with a black circle). (left) and (right) depict two different index pairs, $(P_l,E_l)$ and $(P_r,E_r)$, for the repelling fixed point: $P_l$ and $P_r$ consist of  yellow and red simplices and $E_l$ and $E_r$ consist of red simplices. The intersection $(P_l \cap P_r, E_l \cap E_r)$ is depicted in the middle. Check that this pair is not an index pair because if $e$ denotes a yellow edge, then $F_{\cV}(e) \not\subseteq P_l \cap P_r$. 
Beneath, we depict the barcode that is associated with the zigzag filtration $(P_l,E_l) \supseteq (P_l \cap P_r, E_l \cap E_r ) \subseteq (P_r, E_r)$. Because $(P_l, E_l)$ and $(P_r, E_r)$ are both index pairs for the same repelling fixed point, we would expect the barcode to be full. However, as $(P_l \cap P_r, E_l \cap E_r)$ is not an index pair for the repelling fixed point, its relative homology can change drastically.}
\label{fig:intersection-not-index-pair}
\end{figure} 

To handle this problem, the authors in \cite{DMS2020} introduced a special type of index pair.

\begin{definition}
Let $S$ denote an isolated invariant set, and let $N$ denote an isolating set for $S$. The pair of closed sets $(P,E)$ is an \emph{index pair for $S$ in $N$} if all of the following hold:
\begin{enumerate}
    \item $F_{\cV}(P \setminus E) \subseteq N$
    \item $F_{\cV}(E) \cap N \subseteq E$
    \item $F_{\cV}(P) \cap N \subseteq P$
    \item $S = \inv_{\cV}(P \setminus E)$
\end{enumerate}
\label{def:index_pair_in_N}
\end{definition}

Every index pair in $N$ is also an index pair in the sense of Definition \ref{def:index_pair} (see \cite{DMS2020}). The canonical choice of an index pair for $S$ can be used to obtain a canonical index pair for $S$ in $N$ via the \emph{push forward}. The push forward of a set $A$ in $N$, denoted $\pf_\cV(A,N)$, is given by the set of simplices $\sigma \in N$ for which there exists a path $\rho \; : \; \mathbb{Z}_{[0,n]} \to N$ in $\cV$ where $\rho(0) \in A$ and $\rho(n) = \sigma$. 
If $\cV$ is clear from the context we write $\pf(A,N)$.

\begin{theorem}{ \cite[Theorem 15]{DMS2020} }
Let $S$ be an isolated invariant set under $\cV$, and let $N$ be an isolating set for $S$. The pair $(\pf_{\cV}( \cl(S), N ), \pf_{\cV}( \mo(S), N ) )$ is an index pair in $N$ for $S$. 
\label{thm:push_forward_is_index_pair}
\end{theorem}

Index pairs in $N$ are particularly useful because, unlike standard index pairs, their intersection is guaranteed to be an index pair. For two multivector fields $\cV_1$ and $\cV_2$, an intermediate multivector field is $\cV_1 \overline{\cap} \cV_2$, where $\cV_1 \overline{\cap} \cV_2 := \{V_1 \cap V_2 \; | \; V_1 \in \cV_1, \, V_2 \in \cV_2 \}$.

\begin{theorem}{ \cite[Theorem 10]{DMS2020} } 
Let $(P_1,E_1)$ and $(P_2,E_2)$ denote index pairs in $N$ under $\cV_1$ and $\cV_2$, respectively. The pair $(P_1 \cap P_2, E_1 \cap E_2)$ is an index pair in $N$ under $\cV_1 \overline{\cap} \cV_2$. 
\label{thm:intersecting_index_pairs}
\end{theorem}

Hence, given an index pair $(P_1,E_1)$ in $N$ under $\cV_1$ and an index pair $(P_2,E_2)$ in $N$ under $\cV_2$, we can obtain a relative zigzag filtration where each pair is an index pair under a different multivector field. This zigzag filtration permits capturing a changing Conley index via persistence. We include an example in Figure \ref{fig:intersection-correct}. 

\begin{figure}[htbp]
\centering
\begin{tabular}{ccc}
  \includegraphics[height=32mm]{fig/large-index-pair-opaque.pdf}&
  \includegraphics[height=32mm]{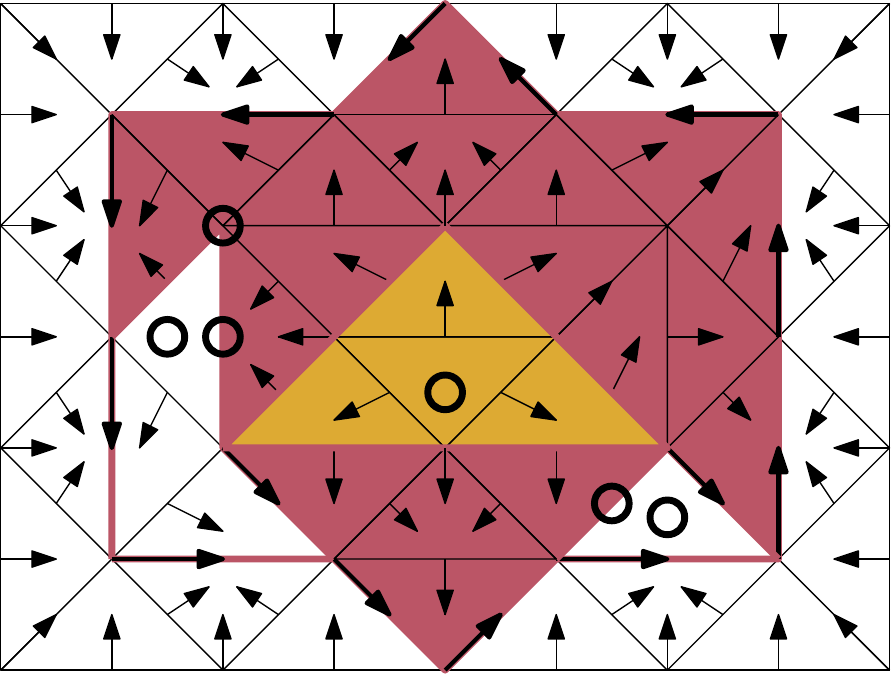}&
  \includegraphics[height=32mm]{fig/resilient-first-index-pair-opaque.pdf}\\
  \multicolumn{3} {l} {
  \begin{tikzpicture}[outer sep = 0, inner sep = 0]
  \draw[fill=dark-gray,outer sep=0, inner sep=0] (0.0,-0.3) -- (13.52,-0.3) -- (13.52,0.1) -- (0.0,0.1) -- cycle;
  \node[align=left,color=white] at (1.3,-0.1) {Dimension: 2};
  \end{tikzpicture} }\\
\end{tabular}
\caption{All three images depict the same multivector field in Figure \ref{fig:intersection-not-index-pair}. The left and the right images depict an index pair in $N$, where $N$ is the entire rectangle. The color convention is the same as in Figure \ref{fig:intersection-not-index-pair}: red and yellow simplices are in $P$, and red simplices are in $E$. Unlike Figure \ref{fig:intersection-not-index-pair}, Theorem \ref{thm:intersecting_index_pairs} implies that that the intersected pair in the middle is an index pair. The persistence barcode, capturing the static Conley index, is depicted below the three images. }
\label{fig:intersection-correct}
\end{figure} 

\subsection{From continuation to filtration}

Now, we show that a continuation of an isolated invariant set $S_1$ to $S_{n+1}$ can be expressed in terms of persistence.
Namely, a corresponding sequence of connecting index pairs $(P_1,E_1)$, $(P_2,E_2)$, \ldots, $(P_n,E_n)$ can be turned into a \emph{zigzag filtration}, 
that is a sequence of pairs $\{(A_i,B_i)\}_{i=1}^{m}$ such that either $(A_{i},B_{i})\subseteq(A_{i+1},B_{i+1})$ or $(A_{i+1},B_{i+1})\subseteq (A_{i},B_{i})$. Ideally, each $(A_i,B_i)$ would be an index pair for some $S_j$ from the initial continuation so as to not introduce spurious invariant sets or Conley indices.
A connecting index pair $(P_i, E_i)$ is an index pair for both $S_i$ under $\cV_i$ and for $S_{i+1}$ under $\cV_{i+1}$.
Thus, $(P_i,E_i)$ and $(P_{i+1},E_{i+1})$ are both index pairs for $S_{i+1}$ under $\cV_{i+1}$.
We will construct auxiliary index pairs for $S_{i+1}$ and then relate $(P_i,E_i)$ and $(P_{i+1},E_{i+1})$ with a zigzag filtration using these auxiliary pairs.
If we can connect all adjacent pairs $(P_i,E_i)$ and $(P_{i+1},E_{i+1})$ with a zigzag filtration, then we can concatenate all of these zigzag filtrations and transform a sequence of connecting index pairs into a larger zigzag filtration. 
The following results are important for achieving this. 

\begin{proposition}{\cite[Proposition 5.2]{LKMW2020}} 
Let $(P,E)$ denote an index pair for $S$. The set $P$ is an isolating set for $S$.
\label{prop:p_isolates}
\end{proposition}
\begin{proposition}
Let $(P,E)$ denote an index pair for $S$ under $\cV$. The pair $(P,E)$ is an index pair for $S$ in $P$ under $\cV$. 
\label{prop:index_pair_in_p}
\end{proposition}
\begin{proof}
First, we observe that $S = \inv_{\cV}(P \setminus E)$ because $(P,E)$ is an index pair. In addition, $F_{\cV}(P) \cap N = F_{\cV}(P) \cap P \subseteq P$ by definition. Since $(P,E)$ is an index pair, it has the property that $F_{\cV}( P \setminus E ) \subseteq P$. In the case of index pairs in $N$, we require that $F_{\cV}( P \setminus E ) \subseteq N = P$, so this case is immediately satisfied. Finally, because $(P,E)$ is an index pair, $F_{\cV}(E) \cap P \subseteq E$. Thus, $F_{\cV}(E) \cap N = F_{\cV}(E) \cap P \subseteq E$. 
\end{proof}
\begin{theorem}
Let $(P_1,E_1)$ and $(P_2,E_2)$ denote index pairs for $S$ in $N$ under $\cV$. The pair $(P_1 \cap P_2, E_1 \cap E_2)$ is an index pair for $S$ in $N$ under $\cV$.
\label{thm:index_pairs_for_same_invariant_set}
\end{theorem}
\begin{proof}
By Theorem \ref{thm:intersecting_index_pairs}, the pair $(P_1 \cap P_2, E_1 \cap E_2)$ is an index pair in $N$ under $\cV$. Hence, it is sufficient to show that $\inv_{\cV}( (P_1 \cap P_2) \setminus (E_1 \cap E_2) ) = S$. Furthermore, $S = \inv_{\cV}( P_1 \setminus E_1 )$ and $S = \inv_{\cV}( P_2 \setminus E_2 )$. Hence, $S \subseteq P_1 \setminus E_1$ and $S \subseteq P_2 \setminus E_2$. Ergo, $S \subseteq P_1 \cap P_2$. In addition, $S \cap E_1 = \emptyset$ and $S \cap E_2 = \emptyset$. Hence, $S \cap (E_1 \cap E_2) = \emptyset$, so it follows that $S \subseteq (P_1 \cap P_2) \setminus (E_1 \cap E_2)$. Thus, $S \subseteq \inv_{\cV}( (P_1 \cap P_2) \setminus (E_1 \cap E_2) )$. Ergo, it remains to be shown that $\inv_{\cV}( (P_1 \cap P_2) \setminus (E_1 \cap E_2) ) \subseteq S$. 

Aiming for a contradiction, assume that there exists an $\sigma \in \inv_{\cV}( (P_1 \cap P_2) \setminus (E_1 \cap E_2) ) \setminus S$. Equivalently, there exists an essential solution $\rho \; : \; \mathbb{Z} \to (P_1 \cap P_2) \setminus (E_1 \cap E_2)$ where $\rho(0) = \sigma$. Because $\rho(\mathbb{Z}) \subseteq P_1 \cap P_2$, but $\rho(\mathbb{Z}) \not\subseteq \inv(P_1 \setminus E_1)$, there must exist an $i_1 \in \mathbb{Z}$ where $\rho(i_1) \in E_1$. Similarly, there must exist an $i_2 \in \mathbb{Z}$ where $\rho(i_2) \in E_2$. 

We claim that for all $i \geq i_1$, $\rho(i) \in E_1$. To contradict, assume that this is not the case. Then there must exist some first $j > i$ where $\rho(j) \not\in E_1$. However, $\rho(j-1) \in E_1$. By definition of an index pair in $N$, if $x \in E_1$ and $y \in F_{\cV}(x) \cap N$, then $y \in E_1$. Hence, since $\rho(j) \not\in E_1$, it follows that $\rho(j) \not\in N$. But by assumption, $\rho(j) \in P_1 \cap P_2 \subseteq N$. Therefore, there is no such $j$, so for all $i > i_1$, we have that $i \in E_1$. The same argument implies that for all $i \geq i_2$, we have that $\rho(i) \in E_2$. 

Thus, it follows that for all $i \geq \max\{i_1, i_2\}$, $\rho(i) \in E_1 \cap E_2$. Ergo, $\rho(\mathbb{Z}) \not\subseteq (P_1 \cap P_2) \setminus (E_1 \cap E_2)$, a contradiction. Hence, no such $\rho$ can exist, which implies that no such $\sigma$ can exist. Thus, $S = \inv_{\cV}( (P_1 \cap P_2) \setminus (E_1 \cap E_2) )$. 
\end{proof}

Now, we move to using these results to translate a sequence of connecting index pairs $\{(P_i,E_i)\}_{i=1}^{n}$ into a zigzag filtration. For $1 < i \leq n$, $(P_{i-1},E_{i-1})$ and $(P_i,E_i)$ are both index pairs for $S_i$. By Proposition \ref{prop:closure-mouth}, the pair $(\cl(S_i),\mo(S_i))$ is an index pair for $S_i$. Hence, a natural approach is to find a zigzag filtration that connects $(P_i,E_i)$ with $(\cl(S_i), \mo(S_i))$ and a zigzag filtration that connects $(P_{i-1},E_{i-1})$ with $(\cl(S_i),\mo(S_i))$. If we can find such zigzag filtrations for all $S_i$, then we can concatenate all of them and obtain a zigzag filtration that connects $(P_1,E_1)$ with $(P_n,E_n)$. We depict the resulting zigzag filtration in Equation \ref{eqn:resulting_continuation}.

\begin{equation}
    (P_1,E_1) \supseteq \ldots \supseteq (\cl(S_2),\mo(S_2)) \subseteq \ldots \subseteq (P_2, E_2) 
        \supseteq \ldots \supseteq (\cl(S_3),\mo(S_3)) \subseteq \ldots (P_{n}, E_{n}) 
    \label{eqn:resulting_continuation}
\end{equation}

We connect $(\cl(S_{i}),\mo(S_i))$ with $(P_i,E_i)$, and $(P_{i-1},E_{i-1})$ connects with $(\cl(S_i),\mo(S_i))$ symmetrically. By Proposition \ref{prop:p_isolates}, $P_i$ is an isolating set for $S_i$. Thus, by Theorem \ref{thm:push_forward_is_index_pair}, $(\pf_{\cV_i}(\cl(S_i), P_i), \pf_{\cV_i}(\mo(S_i), P_i))$ is an index pair for $S_i$ in $P_i$. Proposition \ref{prop:index_pair_in_p} implies that $(P_i,E_i)$ is an index pair for $S_i$ in $P_i$. By Theorem \ref{thm:index_pairs_for_same_invariant_set}, $(P_i \cap \pf_{\cV_i}(\cl(S_i),P_i), E_i \cap \pf_{\cV_i}(\mo(S_i),P_i))$ is an index pair for $S_i$ in $P_i$. Hence, we get the following zigzag filtration:
\begin{multline}
    (\cl(S_i),\mo(S_i)) \subseteq (\pf_{\cV_i}(\cl(S_i), P_i), \pf_{\cV_i}(\mo(S_i), P_i)) \supseteq \\(P_i \cap \pf_{\cV_i}(\cl(S_i),P_i), E_i \cap \pf_{\cV_i}(\mo(S_i),P_i))  \subseteq (P_i,E_i)
    \label{eqn:continuation-filtration}
\end{multline}
Every pair in Equation \ref{eqn:continuation-filtration} is an index pair for $S_i$ under $\cV_i$. Thus, we do not introduce any spurious invariant sets. We can concatenate these filtrations to get Equation \ref{eqn:resulting_continuation}. 

We now analyze the barcode obtained for \ref{eqn:resulting_continuation}. Our main result is Theorem \ref{thm:continuation-barcode}, and it follows immediately from the next two results. 

\begin{lemma}\cite[Lemma 5.10]{LKMW2020}\label{lem:semi-equal_isomorphism}
    Let $(P,E)\subseteq(P',E')$ be index pairs for isolated invariant set $S$ under $\cV$ such that either $P=P'$ or $E=E'$.
    Then the inclusion $i:(P,E)\hookrightarrow(P',E')$ induces an isomorphism in homology.
\end{lemma}

\begin{theorem}
If $(P,E)$ and $(P',E')$ are index pairs for $S$ where $(P',E') \subseteq (P,E)$, then the inclusion induces an isomorphism in the Conley indices. 
\label{thm:inclusion-induces-isomorphism}
\end{theorem}
\begin{proof}
    First, we claim that $(P,E'\cap P)$ and $(P\cup E', E')$ are index pairs for $S$.
    Note that $P\setminus (E'\cap P)= P\setminus E' = (P\cup E')\setminus E'$.
    Thus, $F_\cV(P\setminus(E'))\subseteq F_\cV(P\setminus E)\subseteq P\subseteq P\cup E'$ proves condition \ref{it:index_pair_i} for both pairs.
    Similarly, $S=\inv (P\setminus E)\subseteq \inv (P\setminus E') \subseteq \inv(P'\setminus E')=S$ proves condition \ref{it:index_pair_iii}.
    To see condition \ref{it:index_pair_ii} we have
        $F_\cV(E'\cap P)\cap P\subseteq F_\cV(E')\cap P\subseteq E'\cap P$
        and
        $F_\cV(E')\cap (P\cup E')\subseteq (P'\cap E')\cap (P\cup E')=E'$
        for $(P,E'\cap P)$ and $(P\cup E', E')$, respectively.
    
    Now, consider the following sequence of inclusions
    \begin{center}
      \begin{tikzcd}
        (P,E) \arrow[hookrightarrow, r, "i"]  & 
        (P,E'\cap P)\arrow[hookrightarrow, r, "j"]  &
        (P\cup E', E') \arrow[hookrightarrow, r, "k"]  &
        (P',E').
      \end{tikzcd}
    \end{center}
    Maps $i$ and $k$ induces isomorphisms by Lemma \ref{lem:semi-equal_isomorphism}.
    Similarly, map $j$ induces an isomorphism by the simplicial excision theorem \cite[Theorem 9.1]{Munkres1984}.
\end{proof}

\begin{theorem}
For every $k\geq 0$, the $k$-dimensional barcode of a connecting sequence of index pairs $\{(P_i,E_i)\}_{i=1}^n$ has $m$
bars $[1,n]$ if $\dim H_k(P_1,E_1)=m$.
\label{thm:continuation-barcode}
\end{theorem}

\subsection{Tracking beyond continuation}
In the previous subsection, we showed how to convert a connecting sequence of index pairs into a zigzag filtration. 
Furthermore, we observed that it produces ``full'' barcodes - they have one bar for each basis element of the Conley index that persists for the length of the filtration. 
This change of perspective allows us to generalize our protocol to handle cases when it is impossible to continue.

In particular, we consider Step \ref{protocol-step:persistence} of the protocol. Let $S$ denote an isolated invariant set under $\cV$, and $\cV'$ is an atomic coarsening of $\cV$ where the merged multivector $V$ has the property that $V \cap S \neq \emptyset$ and $V \not\subseteq S$. In such a case, we consider $A := \langle S \cup V \rangle_{\cV'}$ and take $S' = \inv_{\cV'}(A)$. Theorem \ref{thm:impossible_continuation} implies that if $S \neq \inv_{\cV}(A)$, then it is impossible to continue. However, it may be possible to compute persistence in a way that resembles continuation. Let $B := \cl(S) \cup \cl(S')$. Trivially, $B$ is closed. If $B$ is an isolating set for both $S$ and $S'$, then we say that $S$ and $S'$ are \emph{adjacent}. By Theorem \ref{thm:push_forward_is_index_pair}, $(\pf_{\cV}( \cl(S), B ), \pf_{\cV}( \mo(S) , B ))$ is an index pair for $S$ in $B$. Similarly, $(\pf_{\cV'}(\cl(S'),B), \pf_{\cV'}(\mo(S'), B ) )$ is an index pair for $S'$ in $B$. Thus, we can use Theorem \ref{thm:intersecting_index_pairs} to obtain the following zigzag filtration.
\begin{multline}
    (\cl(S),\mo(S)) \subseteq (\pf_{\cV}( \cl(S), B ), \pf_{\cV}( \mo(S) , B )) \\ \supseteq (\pf_{\cV}( \cl(S), B ) \cap \pf_{\cV'}( \cl(S'), B ), \pf_{\cV}( \mo(S) , B ) \cap \pf_{\cV'}( \mo(S') , B ) ) \subseteq \\  (\pf_{\cV'}( \cl(S'), B ), \pf_{\cV'}( \mo(S') , B )) \supseteq (\cl(S'), \mo(S'))
    \label{eqn:general-cont-filtration}
\end{multline}

Suppose that we are iteratively applying Step \ref{protocol-step:continuation} of the Tracking Protocol, finding a sequence of isolated invariant sets where adjacent ones share an index pair, and we terminate with an isolated invariant set $S$ and an index pair $(P,E)$.
We can connect $(P,E)$ with $(\cl(S),\mo(S))$  with techniques from the previous section. 
That is, if $(P,E) \neq (\cl(S),\mo(S))$, then we can find a filtration that connects them:
\begin{multline}
    (P,E) \supseteq (P \cap \pf_{\cV}( \cl(S), P), E \cap \pf_{\cV}(\mo(S),P) ) \subseteq\\ (\pf_{\cV}( \cl(S), P ), \pf_{\cV}( \mo(S), P ) \supseteq (\cl(S), \mo(S_1) )
    \label{eqn:connecting-continuation-filtration}
\end{multline}
We can then concatenate this filtration with the zigzag filtration in Equation \ref{eqn:general-cont-filtration}. 
This effectively completes the Tracking Protocol: when continuation, represented as Step \ref{protocol-step:continuation}, is impossible, we can attempt to apply Step \ref{protocol-step:persistence} and persistence to continue to track. 



In Step \ref{protocol-step:persistence}, we choose to take $S' = \inv_{\cV'}(A)$. In practice, there may be many isolated invariant sets under $\cV'$ that are adjacent to $S$. However, our choice of $S'$ is canonical.
\begin{proposition}
Let $S'$ denote an isolated invariant set under $\cV'$ that is obtained from applying Step \ref{protocol-step:persistence} of the Tracking Protocol to the isolated invariant set $S$ under $\cV$. If $S''$ is an isolated invariant set under $\cV'$ where $S \subseteq S''$, then $S' \subseteq S''$.
\end{proposition}
\begin{proof}
    By Proposition \ref{prop:invariant_part_of_convex_vcomp}, set $S''$ is convex and $\cV'$-compatible.
    Since $A$ is the minimal convex and $\cV'$-compatible set containing $S$ we get that $S\subseteq A\subseteq S''$. By definition, $S' = \inv_{\cV'} (A)$, so $S' \subseteq A$ $\subseteq S''$. 
\end{proof}

\subsection{Strategy for Step \ref{protocol-step:impossible} of the Tracking Protocol}
\label{sec:strategy}

We briefly consider Step \ref{protocol-step:impossible} of Tracking Protocol, which is the case when it is impossible to continue from $S$ to some $S'$, and $B = \cl(S) \cup \cl(S')$ does not isolate both $S$ and $S'$. In such a case, we do not have two index pairs in the same isolating set, so we cannot use Theorem \ref{thm:push_forward_is_index_pair} to obtain index pairs in a common $N$. 
Thus, we cannot use Theorem \ref{thm:intersecting_index_pairs}, which permits us to intersect index pairs and guarantee that the resulting pair is an index pair under a known multivector field. Recall that $A = \langle S \cup V \rangle_{\cV'}$. A natural choice is to take $S' = \inv_{\cV'}(A)$, and consider the zigzag filtration
\begin{equation}
    (\cl(S),\mo(S)) \supseteq ( \cl(S) \cap \cl(S'), \mo(S) \cap \mo(S') ) \subseteq (\cl(S'),\mo(S')).
\end{equation}
It is easy to construct examples where $( \cl(S) \cap \cl(S'), \mo(S) \cap \mo(S') )$ is not an index pair under any natural choice of multivector field (see Figure \ref{fig:intersection-not-index-pair}). However, this approach may work well in practice. We leave a thorough examination to future work.

%% file: 6conclusion.tex
\section{Conclusion}

We conclude with discussing briefly some directions for future work. In Step \ref{protocol-step:impossible} of the Tracking Protocol, there is a canonical choice of $S'$. But, as there is no common isolating set for $S$ and $S'$, we cannot presently say anything about the persistence of the Conley index from $S$ to $S'$. Is it possible to compute the Conley index persistence here in a controlled way? Can we meaningfully compute persistence for a different choice of invariant set?  Investigation and experiments are likely needed to determine the most practical course of action in this case. 

\section*{Acknowledgment}
This work is partially supported by NSF grants CCF-2049010, CCF-1839252, Polish National Science Center under Ma\-estro Grant 2014/14/A/ST1/00453, Opus Grant 2019/35/B/ST1/00874 and Preludium Grant 2018/29/N/ST1/00449.